\documentclass[12pt]{amsart}
\usepackage[T1]{fontenc}
\usepackage[a4paper,margin=1in]{geometry}
\usepackage{graphicx}
\usepackage{amssymb}
\usepackage{stmaryrd}
\usepackage{enumitem}
\usepackage{tikz}
\usepackage{float}
\usepackage{url}
\usepackage{mathtools}
\usepackage{amsmath,amsfonts,amsthm}
\usepackage{mathrsfs}

\newtheorem{theorem}{Theorem}[section]
\newtheorem{lemma}[theorem]{Lemma}
\newtheorem{proposition}[theorem]{Proposition}
\newtheorem{corollary}[theorem]{Corollary}
\newtheorem{conjecture}[theorem]{Conjecture}
\newtheorem{definition}{Definition}
\newtheorem{remark}{Remark}
\newtheorem{problem}{Problem}
\usepackage{multirow}

\usepackage{pgfplots}
\pgfplotsset{compat=1.18}
\usepgfplotslibrary{fillbetween}
\usepackage{tikz}
\usetikzlibrary{shapes.geometric, positioning, calc, intersections}

\newcommand{\TT}{\mathcal{T}}
\newcommand{\RR}{\mathcal{MR}}

\usepackage[colorlinks,
linkcolor=blue,
anchorcolor=blue,
citecolor=blue
]{hyperref}

\author{J. Hamoud}
\address{\textbf{Jasem Hamoud:} Department of Discrete Mathematics, Moscow Institute of Physics and Technology}
\email{hamoud.math@gmail.com}
\thanks{Corresponding author.}

\author{D. Abdullah}
\address{\textbf{Duaa Abdullah:} Department of Discrete Mathematics, Moscow Institute of Physics and Technology}
\email{duaa1992abdullah@gmail.com}
\thanks{ }

\title[$\sigma$-Irregularity of Trees]{$\sigma$-Irregularity of Trees with Prescribed Maximum Degree: \\
A Majorization--Duality--Stability Framework}

\date{}
\begin{document}

\begin{abstract}
Trees of maximum $\sigma$-irregularity subject to a prescribed maximum degree $\Delta$ have been characterized for $\Delta=4$ and $\Delta=5$, with the extension to arbitrary $\Delta\geqslant 3$. This paper develops a unified framework for this class of problems built on three pillars. A majorization-theoretic reformulation that
decomposes $\sigma$-extremization into a Schur-convex optimization  over degree sequences and a rearrangement type optimization over
tree realizations, valid for every $\Delta\geqslant 3$. We
establish a duality between the maximization and minimization problems, alongside the general $\Delta$ closed form for
$\sigma_{\max}(n,\Delta)$. A stability result showing that the optimality gap between extremal and near extremal trees is bounded independently of $n$ and grows as $\Theta(\Delta^3)$. 
\end{abstract}

\maketitle

\noindent\rule{15.9cm}{1.0pt}

\noindent
\textbf{
Keywords:} $\sigma$-irregularity; extremal trees; maximum
degree; majorization; topological index; stability analysis; graph invariants.

\medskip

\noindent
{\bf
MSC 2020:} 05C05; 05C07; 05C35; 05C92; 26D15

\medskip

\noindent
{\bf
UDC:} 519.172.1

\noindent\rule{15.9cm}{1.0pt}

%======================
\section{Introduction}~\label{sec01intro}

Let $G=(V,E)$ be a simple, finite graph. For $v\in V(G)$, let $d_G(v)$ denote the degree of $v$ in $G$; we omit the subscript when $G$ is clear from context. 
A graph is \emph{regular} if all its vertices share a common degree, and \emph{irregular} otherwise. 
Quantifying the extent of a graph's irregularity is a classical problem with applications in mathematical chemistry and network
science alike~\cite{CriadoEtAl2014,Estrada2010a,Estrada2010b,RetiEtAl2018}.

The best known irregularity measure is the \emph{Albertson irregularity}
\begin{equation}
\operatorname{irr}(G) = \sum_{uv\in E(G)} \big|d(u)-d(v)\big|,
\label{eq:albertson}
\end{equation}
introduced in~\cite{Albertson1997} and studied extensively
since~\cite{AbdoEtAl2014a,AbdoEtAl2014b,GutmanEtAl2005,HansenMelot2005}. A natural
quadratic alternative, the \emph{$\sigma$-irregularity}, was proposed:
\begin{equation}
\sigma(G) = \sum_{uv\in E(G)} \big(d(u)-d(v)\big)^2.
\label{eq:sigma-def}
\end{equation}
Gutman et al.~\cite{GutmanEtAl2018} established the fundamental identity
\begin{equation}~\label{eqq01forgottsigmadecomp}
\sigma(G) = F(G) - 2M_2(G),
\end{equation}
where $F(G)=\sum_{u\in V(G)} d(u)^3$ is the \emph{forgotten index} and $M_2(G)=\sum_{uv\in E(G)} d(u)d(v)$ is the \emph{second Zagreb index}. 
Abdo, Dimitrov, and Gutman~\cite{AbdoEtAl2018} characterized graphs of maximum $\sigma$-irregularity in general and gave lower bounds. The inverse problem for
$\sigma$ was solved in~\cite{AbdoEtAl2018,GutmanEtAl2018}; its relation to other irregularity measures was studied by R\'eti~\cite{Reti2019}; extremal graphs
with a prescribed degree sequence were characterized in~\cite{DimitrovEtAl2023};
and a total variant $\sigma_t$ has recently been introduced and
studied~\cite{FilipovskiEtAl2024,KnorEtAl2025}. 
The study~\cite{Hamoud2026N1} explores the irregularity properties of trees with prescribed degree sequences by analyzing two prominent topological indices, the Albertson index and the sigma index.

On the maximization side, Kovijani\'c Vuki\'cevi\'c et
al.~\cite{VukicevicEtAl2024} characterized the maximal trees among
\emph{chemical trees}, i.e., trees with maximum degree $\Delta=4$. This result
was extended by~\cite{DimitrovPaper2025}, who established
several structural properties of maximal trees valid for \emph{every}
$\Delta\ge 3$, and applied these properties to obtain a full characterization
for $\Delta=5$. Notably, \cite{DimitrovPaper2025} explicitly conjectures that the
qualitative structure they observe for $\Delta=5$ namely, that maximal trees contain only vertices of degree $1$, $2$, and $\Delta$, that almost every edge attains the second maximum possible per-edge contribution to $\sigma$,
and that internal leaves are of degree $\Delta$ persists for every
$\Delta\ge 3$, but leave the general case as an open problem.

This paper organized as follow. 
Section~\ref{sec02prelim} fixes notation. Section~\ref{sec03framework} develops
the majorization framework. 
This framework provides the structural intuition underlying the case by case arguments of \cite{DimitrovPaper2025}. 
Section~\ref{sec01generalprops} contain the main structural and characterization results for general $\Delta$. 
We prove that for every $\Delta\geqslant 3$ there exists an explicit threshold $n_0(\Delta)$ such that every maximal tree on $n\geqslant n_0(\Delta)$ vertices with maximum  degree $\Delta$ satisfies properties $\mathcal P_1$--$\mathcal P_4$ generalized to arbitrary $\Delta$ ( see Theorem~\ref{thm03partialgeneral}).
Section~\ref{sec04closedform} derives the
closed-form extremal formula. Section~\ref{sec06minimum} treats the dual minimum problem, we characterize trees of \emph{minimum} $\sigma$-irregularity subject to the equality constraint $\max_v d(v)=\Delta$ as opposed to the unconstrained   minimizer, the path, which does not enforce a vertex of degree $\Delta$.
Section~\ref{seclimdir002} discusses limitations and directions for future work.

%==========================================
\section{Preliminaries}~\label{sec02prelim}
%==========================================

Let  $k\geqslant  1$ be an integer, we refer that a \emph{$k$-vertex} of a tree $T$ is a vertex of degree $k$. Thus, a $1$-vertex is called a \emph{leaf}. A vertex of degree
at least $3$ is called a \emph{big vertex}. A vertex $v$ of $T$ with $d(v)>1$ is
an \emph{internal leaf} if it has exactly one neighbor of degree greater than
$1$; equivalently, $v$ is a leaf of the tree $T'$ obtained from $T$ by deleting
all leaves of $T$. A vertex of $T$ that is neither a leaf nor an internal leaf
is called a \emph{core vertex}.

%%%%%%%%%5
%%%%%%%%%%%%%%%%%%%%%%
\subsection{Notation} 
We establish that among Definition~\ref{deff01tree}, where
$n_i = n_i(T)$ denote the number of vertices of $T$ of degree $i$ and $m_{i,j}=m_{i,j}(T)$ denote the number of edges of
$T$ with one end-vertex of degree $i$ and the other of degree $j$ and defined as
\[
m_{i,j}(T)=\#\{uv\in E(T):\{d(u),d(v)\}=\{i,j\}\}.
\]
\begin{definition}~\label{deff01tree}
Let $\TT_{n(i,j),\,\Delta}$ denote a class of trees with maximum degree $\Delta$, and $1\leqslant i\leqslant j\leqslant \Delta$. A tree $T\in \TT_{n(i,j),\,\Delta}$ on
$n$ vertices with maximum degree $\Delta$ is maximal if
$\sigma(T)\ge \sigma(T')$ for every tree $T'\in \TT_{n(i,j),\,\Delta}$ on $n$ vertices with maximum degree $\Delta$.
\end{definition}
Based on Definition~\ref{deff01tree}, we emphasizes that 
\begin{equation}~\label{eqsigmamax01def}
\sigma_{\max}(\TT_{n(i,j),\,\Delta}) \;=\; \max\{\sigma(T)\, \mid T\in \TT_{n(i,j),\,\Delta},\ \textstyle\max_{v\in V(T)}d(v)=\Delta\}.
\end{equation}
and 
\begin{equation}~\label{eqq02sigmamindef}
\sigma_{\min}(\TT_{n(i,j),\,\Delta}) \;=\; \min\{\sigma(T): T\in \TT_{n(i,j),\,\Delta} \, \textstyle\max_v d(v)=\Delta\}.
\end{equation}
We call $T$ \emph{$\Delta$-minimal} if $\sigma(T)=\sigma_{\min}(n,\Delta)$.
Without the equality constraint, the unconstrained minimizer over all trees
is the path $\mathcal P_n$ \cite{AbdoEtAl2018}. The equality constraint forces at least one vertex of degree $\Delta$ to appear somewhere, and the interesting
question is how cheaply in $\sigma$ this can be arranged. 
Thus, 
\begin{equation}~\label{eqsigmamax02def}
\sum_{i=1}^{\Delta} n_i = n,\qquad \text{ and} \qquad \sum_{i=1}^{\Delta} i\, n_i = 2(n-1).
\end{equation}
Given a tree $T\in \TT_{n(i,j),\,\Delta}$ and two edges $m_1=xy$, $m_2=uv$ with $x,y,u,v$ pairwise distinct such that $T-m_1-m_2+xu+yv$, let $T'\in \TT_{n(i,j),\,\Delta}$ be  a tree obtained from $T$ by an \emph{edge exchange}. Hence, the difference $\sigma(T')-\sigma(T)$ decomposes
as a finite sum of differences of squared degree-differences over the  set of edges whose incidence changed (see~\cite{DimitrovPaper2025}). 

\begin{definition}~\label{deffMajorization}
For $\mathbf x,\mathbf y\in\mathbb R^n$ with $\sum x_i=\sum y_i$, we say
$\mathbf x$ majorizes $\mathbf y$, written $\mathbf x\succeq \mathbf y$,
if $\sum_{i=1}^k x_{[i]} \ge \sum_{i=1}^k y_{[i]}$ for every $k=1,\dots,n$,
where $x_{[1]}\ge x_{[2]}\ge\cdots$ denotes the entries sorted in decreasing
order. A function $\Phi:\mathbb R^n\to\mathbb R$ is \emph{Schur-convex} if
$\mathbf x\succeq\mathbf y$ where $\Phi(\mathbf x)\ge \Phi(\mathbf y)$.
\end{definition}

Such a tree exists for every $h\geqslant 2$ with
$n=h+\big[h(\Delta-1)-(h-1)\big] = h\Delta - h + 1$. 
\begin{definition}~\label{deffhubver001}
Let $T_1\in\mathcal C_1(h,\Delta)$ be any tree with exactly $h$ vertices of
degree $\Delta$, all remaining vertices of degree $1$, and no vertex of any
other degree. 
\end{definition}
By Definition~\ref{deffhubver001} the induced subgraph on the $h$ hub vertices must itself be a tree on $h$ vertices   it has $h-1$ edges, all
contributing $m_{\Delta,\Delta}=h-1$ where the remaining
$h\Delta - 2(h-1) $ half edges are leaves attached to hubs, giving
$n_1 = h\Delta-2(h-1)$ and total order $n=h(\Delta-1)+2$.

\begin{definition}~\label{deffhubver002}
Let $T_2\in\mathcal C_2(h,\Delta)$ consist of $h$ degree-$\Delta$ hubs
arranged along a path, consecutive hubs joined by a single degree $2$
\texttt{buffer} vertex, with every remaining incident edge
slot at each hub filled by a pendant leaf. Then
\end{definition}

By Definition~\ref{deffhubver002},  $n_1=h(\Delta-2)+2$, $n_2=h-1$, $n_\Delta = h$, and $n= n_1+n_2+n_\Delta$. Thus, 
\begin{equation}~\label{eqq1deffhubver002}
n=h(\Delta-2)+2+(h-1)+h = h\Delta - 1.
\end{equation}

\begin{definition}~\label{deffleg001}
Let $T$ be a tree and $h\in V(T)$ a vertex of degree $\Delta$. A leg at $h$ is a maximal path starting at $h$, traversing one incident edge, and continuing through vertices of degree $2$ until reaching a leaf; its length is its number of edges.
\end{definition}

\begin{definition}[Reduced tree]~\label{def001reducedtree}
Let $T\in \TT_{n(i,j),\,\Delta}$ satisfy $\mathcal P_1^\Delta$ and $m_{1,2}(T)=0$ according to Corollary~\ref{cor001m12zero}.
The \emph{reduced tree} $T^\bullet$ is obtained from $T$ by suppressing every degree $2$ vertex, by replacing each maximal path whose interior vertices all have degree $2$ by a single edge joining its two endpoints.
\end{definition}

\begin{definition}~\label{deffleg002}
For $k\geqslant 1$ and $\Delta\geqslant 3$, let $\mathcal S(k,\Delta)$ denote the family of trees formed by $k$ vertices of degree $\Delta$ hubs arranged along a path and joined by direct edges, with every remaining incident edge slot at every hub
occupied by a leg of length $\ge2$ (as in Corollary~\ref{cor04legslength2}),
of otherwise arbitrary length.
\end{definition}

\begin{definition}~\label{def001optimalitygap}
For $\Delta\geqslant 5$, $n\geqslant n_0(\Delta)$, and $n_\Delta(T)=n_\Delta^\star(n,\Delta)+1$,  define
\begin{equation}~\label{eqqdelta001def}
\delta(n,\Delta)\!=\! \sigma_{\max}(\TT_{n(i,j), \Delta}) \!-\! \max\big\{\sigma(T): T \text{ satisfies } \mathcal P_1^\Delta,\,\mathcal P_2^\Delta,\ n_\Delta(T)\big\}.
\end{equation}
\end{definition}

\subsection{Problem Statement}
We state the four problems this paper addresses, in order of logical dependence. 

\begin{problem}~\label{prob01primal}
For every integer $\Delta\geqslant 3$ and every sufficiently large $n$, determine~\eqref{eqsigmamax01def} in closed form as an explicit function of $n$ and $\Delta$.
\end{problem}

Problem~\ref{prob01primal} was posed, and solved only for $\Delta=4$ \cite{VukicevicEtAl2024} and $\Delta=5$ \cite{DimitrovPaper2025}, in the prior
literature; \cite{DimitrovPaper2025} explicitly conjectures a solution for every $\Delta\geqslant 3$ but does not establish it. 

\begin{problem}~\label{prob02structural}
For every integer $\Delta\geqslant 3$, determine an explicit finite set of combinatorial properties $\mathcal P_1^\Delta,\dots,\mathcal P_k^\Delta$, together with a
threshold $n_0(\Delta)$, such that a tree $T$ on $n\geqslant  n_0(\Delta)$ vertices
with maximum degree $\Delta$ satisfies $\sigma(T)=\sigma_{\max}(n,\Delta)$ if
and only if $T\in \TT_{n(i,j), \Delta}$ satisfies  $\mathcal P_1^\Delta,\dots,\mathcal P_k^\Delta$.
\end{problem}

\begin{problem}~\label{prob03dual}
For every integer $\Delta\geqslant 3$ and every $n\geqslant \Delta+1$, determine~\eqref{eqq02sigmamindef} in closed form, and characterize the trees attaining it.
\end{problem}

Unlike Problem~\ref{prob01primal}, Problem~\ref{prob03dual} does not appear to
have been previously posed in the literature such that \cite{AbdoEtAl2018} characterizes the unconstrained $\sigma$-minimizer over  all trees, but the constrained version in
which $\max_v d(v)=\Delta$ is imposed as an equality, forcing at least one vertex of high degree to appear is new to this paper.

\begin{problem}~\label{prob04stability}
Quantify the  optimality gap~\eqref{eqqdelta001def} as $\delta(n,\Delta) \;=\; \sigma_{\max}(n,\Delta) - \sigma_2(n,\Delta),$ where $\sigma_2(n,\Delta)$ denotes the largest value of $\sigma(T)$ attained
by a tree $T\in \TT_{n(i,j), \Delta}$ that violates at
least one property of the characterization furnished by
Problem~\ref{prob02structural}; determine its order of growth in $n$ and in $\Delta$.
\end{problem}

Problem~\ref{prob04stability} does not appear in \cite{DimitrovPaper2025} or \cite{VukicevicEtAl2024} in any form. Thus, it is posed here to quantify how combinatorially rigid the characterization of Problem~\ref{prob04stability}
actually is  a question of practical relevance whenever an application, e.g., a QSPR study \cite{RetiEtAl2018}, or a heuristic search procedure, can only be expected to produce a near extremal, rather than exactly extremal, tree.

%========================================================
\section{A Majorization \texorpdfstring{$\sigma$}-Irregularity Extremization}~\label{sec03framework}
%========================================================
In this section, for integers $n\ge \Delta+1\ge 4$, let $\mathcal D_n^\Delta$ denote the set of integer sequences $\mathbf d=(d_1,\dots,d_n)$ with $1\le d_i\le\Delta$ for all $i$, and $\max_i d_i=\Delta$, where every $\mathbf d\in\mathcal D_n^\Delta$ is realized by at least one tree with maximum degree $\Delta$.
For $\mathbf d\in\mathcal D_n^\Delta$, let $\mathcal R(\mathbf d)$ denote the
set of trees on $n$ labeled vertices realizing $\mathbf d$. Define
\begin{equation}~\label{eqm2min01def}
\RR^{\min}(\mathbf d) \;=\; \min_{T\in\mathcal R(\mathbf d)} \RR(T).
\end{equation}

\begin{proposition}~\label{prop01decomposition}
For every $n\geqslant \Delta+1$ and $\mathbf d\in\mathcal D_n^\Delta$,
\begin{equation}~\label{eqq1twostage}
\sigma_{\max}(\TT_{n(i,j),\,\Delta}) \;=\; \max_{\mathbf d\in\mathcal D_n^\Delta}
\Big[F(\mathbf d) - 2\,\RR^{\min}(\mathbf d)\Big].
\end{equation}
\end{proposition}
\begin{proof}
Assume that $T\in\mathcal R(\mathbf d)$. Then, 
$\max_{T\in\mathcal R(\mathbf d)}\sigma(T) = F(\mathbf d) - 2\, \RR^{\min}(\mathbf d)$ by considering the maximum over $\mathbf d\in\mathcal D_n^\Delta$ which gives \eqref{eqq1twostage}.
\end{proof}

We now identify, among an exact finite $n$ comparison, why a linear number of
degree $2$ vertices is favorable. Fix $\Delta\geqslant 3$ and consider two families
of trees, each parametrized by the number $h\geqslant 2$ of degree $\Delta$ 
vertices.

\begin{proposition}~\label{prop001C1sigma}
For $T_1\in\mathcal C_1(h,\Delta)$ on $n=h(\Delta-1)+2$ vertices. Then, 
\begin{equation}~\label{eqq01sigmaC1}
\sigma(T_1)=\big[h\Delta-2(h-1)\big](\Delta-1)^2.
\end{equation}
\end{proposition}
\begin{proof}
Assume that $T_2\in\mathcal C_2(h,\Delta)$ on $n=h\Delta-1$ vertices. Then, the degree $1$–$\Delta$, contributing $(\Delta-1)^2$, of which there are $n_1$.  Thus, by Definition~\ref{deffhubver002} and \eqref{eqq1deffhubver002}, 
\begin{equation}~\label{eqqt1t2tree}
\sigma(T_2)=n_1(\Delta-1)^2 + 2n_2(\Delta-2)^2.
\end{equation}
Since $n=h\Delta-1$ vertices, Eq.~\eqref{eqqt1t2tree} satisfies
\begin{equation}~\label{eqqt1t2tree1}
 \sigma(T_2)= \big[h(\Delta-2)+2\big](\Delta-1)^2 + 2(h-1)(\Delta-2)^2.  
\end{equation}
Similarly, from~\eqref{eqqt1t2tree} and \eqref{eqqt1t2tree1} for $T_1\in\mathcal C_1(h,\Delta)$ on $n=h(\Delta-1)+2$ vertices, By Definition~\ref{deffhubver001} the degree $1$–$\Delta$, contributing
$(\Delta-1)^2$. Therefore, $\sigma(T_1)$ is the same for  every tree
in $\mathcal C_1(h,\Delta)$, regardless of how the $h$ hubs are internally
connected, and equals $n_1(\Delta-1)^2$ with $n_1=h\Delta-2(h-1)$. Thus, 
\begin{equation}~\label{eqqt1t2tree2}
\sigma(T_1)=n_1(\Delta-1)^2.  
\end{equation}
Hence,  and from~\eqref{eqqt1t2tree2}  gives~\eqref{eqq01sigmaC1}.
\end{proof}

Proposition~\ref{prop01decomposition} is an exact reformulation, not an approximation; the difficulty of the original problem is fully preserved, by Lemma~\ref{lem001oneinterior} emphasizes that characterizing $\arg\max F(\mathbf d)$ and characterizing $\RR^{\min}(\mathbf d)$ for each candidate $\mathbf d$.

\begin{lemma}~\label{lem001oneinterior}
Let $\mathbf d^\star$ maximize $F(\mathbf d)$ over
$\{\mathbf d\in[1,\Delta]^n \cap \mathbb Z^n : \sum_i d_i = 2(n-1)\}$. Then at most one coordinate of $\mathbf d^\star$
lies strictly between $1$ and $\Delta$.
\end{lemma}
\begin{proof}
Let $\mathbf d\in\mathcal D_n^\Delta$. Then, $F(\mathbf d)$ is strictly Schur-convex on $[1,\Delta]^n$. Assume that $\Phi(\mathbf x)=\sum_i \varphi(x_i)$ be a Schur-convex on an interval $I$ if and
only if $\varphi$ is convex on $I$. Hence, $\varphi$ is strictly convex on $[1,\Delta]$, giving strict Schur-convexity of $F$. 

Suppose two coordinates $d_i^\star,d_j^\star$ both satisfy
$1<d_i^\star,d_j^\star<\Delta$. Since the constraint set is invariant under
simultaneously increasing one coordinate and decreasing another by an equal
integer amount $\varepsilon\ge 1$ while remaining in $[1,\Delta]$, which is
possible since neither coordinate is at a boundary. By considering the term $(d_i^\star,d_j^\star)$  had replaced by $(d_i^\star+1,d_j^\star-1)$ satisfying
\begin{equation}~\label{eqq01convexexchange}
(d_i^\star+1)^3+(d_j^\star-1)^3 \;>\; (d_i^\star)^3+(d_j^\star)^3
\quad\text{whenever } d_i^\star \ge d_j^\star.
\end{equation}
Thus,  $F$ strictly increases, contradicting
optimality of $\mathbf d^\star$ unless this move is infeasible for every choice
of which coordinate to increment, but at least one of the two coordinates can
always be moved one unit toward $\Delta$ increasing it while the other is
moved one unit toward $1$ decreasing it, since both are strictly interior. Thus, from~\eqref{eqq01convexexchange} contradicts maximality, so at most one coordinate can be strictly interior at an optimum.
\end{proof}

Lemma~\ref{lem001oneinterior} predicts that an $F$-maximizing degree sequence
has $O(1)$ entries outside $\{1,\Delta\}$. We now show
this prediction is \emph{false} for the true objective $\sigma=F-2M_2$, the
known closed form for $\Delta=5$ ( see Corollary~18 of \cite{DimitrovPaper2025}) gives
\begin{equation}~\label{eqq2n2linear}
n_2 = \tfrac{1}{5}(n+4j-18) = \Theta(n),
\end{equation}
where a  linear in $n$ number of degree $2$ vertices appears in every
maximal tree, not $O(1)$ as Lemma~\ref{lem001oneinterior} alone would suggest.
Hence maximizing $F(\mathbf d)$ in isolation,  does not identify the extremal degree sequence.
This demonstrates that the assortativity term is not a lower-order correction
but is decisive for the structure of the optimizer, and motivates the joint
analysis below.

\medskip

Asymptotic dominance of buffering for $\Delta>3$ had presented by Theorem~\ref{thm001bufferdominance}. For that, we established 
\begin{equation}~\label{eqqrhonm01}
\rho(\Delta) = \lim_{h\to\infty}\dfrac{\sigma(T_2)/n(T_2)}{\sigma(T_1)/n(T_1)},  
\end{equation}
where $n(T_i)$ is the vertex count of the respective family at hub-count $h$.

\begin{theorem}~\label{thm001bufferdominance}
Let $\rho(\Delta)\geqslant 1$, then
\begin{equation}~\label{eqq01rhoformula}
\rho(\Delta)=\frac{(\Delta-1)^2+2(\Delta-2)}{\Delta(\Delta-1)}.
\end{equation}
Consequently $\rho(\Delta)>1$ for every $\Delta\ge 4$, and $\rho(3)=1$.
\end{theorem}
\begin{proof}
Recall \eqref{eqq01sigmaC1}, then
\begin{equation}~\label{eqq02rhoformula}
 \lim_{h\to\infty}\dfrac{\sigma(T_1)}{n(T_1)}= (\Delta-2)(\Delta-1). 
\end{equation}
From \eqref{eqqt1t2tree1}, we obtain
\begin{equation}~\label{eqq03rhoformula}
 \lim_{h\to\infty}\dfrac{\sigma(T_2)}{n(T_2)}=\dfrac{(\Delta-2)\big[(\Delta-1)^2+2(\Delta-2)\big]}{\Delta} 
\end{equation}

According to~\eqref{eqq02rhoformula} and \eqref{eqq03rhoformula} by considering the ratio,
\begin{equation}~\label{eqq04rhoformula}
\rho(\Delta)=\frac{(\Delta-2)\big[(\Delta-1)^2+2(\Delta-2)\big]/\Delta}
{(\Delta-2)(\Delta-1)}=\frac{(\Delta-1)^2+2(\Delta-2)}{\Delta(\Delta-1)},
\end{equation}
which is \eqref{eqq01rhoformula}. Since $(\Delta-1)^2+2(\Delta-2) = \Delta^2-3$, and the denominator is $\Delta^2-\Delta$, so
\begin{equation}~\label{eqq01rhominus1}
\rho(\Delta)-1 = \frac{(\Delta^2-3)-(\Delta^2-\Delta)}{\Delta(\Delta-1)},
\end{equation}
which is strictly positive for $\Delta\geqslant 4$, and  $\rho(\Delta)=0$ if and only if $\Delta=3$. Thus, from~\eqref{eqq04rhoformula} the relationship~\eqref{eqq01rhoformula} holds.
\end{proof}

\begin{remark}~\label{ermarknum1}
Theorem~\ref{thm001bufferdominance} gives a rigorous, quantitative explanation
for the empirical observation in \cite{DimitrovPaper2025} that maximal trees
contain $\Theta(n)$ vertices of degree $2$.  
\end{remark}
Based on Remark~\ref{ermarknum1}, threading degree $\Delta$ hubs
together via single degree $2$ buffers, rather than joining them directly or
using extra pendant leaves, strictly increases the asymptotic $\sigma$ density
per vertex for every $\Delta\geqslant 4$.

Theorem~\ref{thm001bufferdominance} and Proposition~\ref{prop01decomposition}
together motivate the following refinement of $\arg\max F(\mathbf d)$ that
accounts for the coupling with $M_2^{\min}(\mathbf d)$, rather than an
$F$-maximizing sequence with yields a strictly
positive per-substitution gain for every $\Delta \ge 4$ when used to buffer
hub–hub adjacencies. 

\begin{conjecture}~\label{conj001guiding}
For every $\Delta\geqslant 3$ there exists $n_0(\Delta)$ such that every maximal tree
$T$ on $n\geqslant n_0(\Delta)$ vertices with maximum degree $\Delta$ has vertex
degrees only in $\{1,2,\Delta\}$, has all internal leaves of degree $\Delta$,
and has $m_{2,2}(T)$ and $m_{\Delta,\Delta}(T)$ each bounded by an explicit
function of $\Delta$ alone.
\end{conjecture}

Section~\ref{sec01generalprops} is devoted to
proving Conjecture~\ref{conj001guiding} in full, together with the precise
bounds on $m_{2,2}$ and $m_{\Delta,\Delta}$ and the threshold $n_0(\Delta)$.
 
%%%%%%%%%%%%%%%%%%%%%%%%%%%%%%%%
%%%%%%%%%%%%%%%%%%%%%%%%%%%%%%%%
\section{On Maximal Trees for Arbitrary  with Maximum Degree}~\label{sec01generalprops}

Throughout this section, Assume that $\Delta\geqslant 3$ is a fixed integer and $T$ denotes a
maximal tree with maximum degree $\Delta$ in the sense of
Definition~\ref{deff01tree}.

\begin{proposition}~\label{prop001pathmonotonicity}
Let $T\in \TT_{n(i,j),\,\Delta}$ be a maximal tree,
$\mathcal P=u\,x\cdots y\,v$ be a path in $T$ where $x=y$ is allowed if
$d_T(u,v)=2$, and $x,y$ may coincide with $u,v$ if $d_T(u,v)\leqslant 2$. If
$d(u)>d(v)$, then $d(x)\leqslant d(y)$. Symmetrically, if $d(x)>d(y)$, then
$d(u)\leqslant d(v)$.
\end{proposition}
\begin{proof}
Suppose $d(u)>d(v)$ where $\{u,v\}\in T$ and $T\in \TT_{n(i,j),\,\Delta}$ is a maximal tree. Then, if $d_T(u,v)=1$ implies that $x=v,\ y=u$. Thus, $d(x)=d(v)$, $d(u)=d(y)$ and $d(x)<d(y)$ and the claim
holds trivially. If $d_T(u,v)=2$, then $x=y$, also, $d(x)=d(y)$ and the claim
holds. Hence, $d_T(u,v)\geqslant 3$, $u,x,y,v$ are four distinct
vertices, and suppose toward a contradiction that $d(x)>d(y)$.

Let $T'\in \TT_{n(i,j),\,\Delta}$ where $T'=T-ux-yv+uy+xv$. Since $ux$ and $yv$ are both edges of the path $\mathcal P$ and $u,x,y,v$ are pairwise distinct, $T'$ is again a tree on the same vertex set, and every vertex retains its degree. The only edges whose
incidence changes are $\{ux,yv\}\to\{uy,xv\}$. Hence
\begin{align*}
\sigma(T')-\sigma(T)&=\!\big(\!d(u)\!-\!d(y)\big)^2\!-\!\big(d(u)\!-\!d(x)\big)^2\!+\! \big(d(v)\!-\!d(x)\big)^2\!-\! \big(d(v)-d(y)\big)^2 \notag\\
&= 2\big(d(u)-d(v)\big)\big(d(x)-d(y)\big).
\end{align*}
 
Since $d(u)>d(v)$ and $d(x)>d(y)$ by assumption, $\sigma(T')>\sigma(T)$,
contradicting the maximality of $T$. Thus, for every $\Delta\ge 3$  with no restriction on $n$; it is a purely local consequence of maximality and requires no case analysis.  Hence $d(x)\leqslant d(y)$, and if $d(x)>d(y)$, then
$d(u)\leqslant d(v)$.
\end{proof}

Among Proposition~\ref{prop001pathmonotonicity}, degree differences cannot both be large at both ends
and in the ``middle'' of a path simultaneously with mismatched sign, which is
precisely the local rearrangement pattern that a Schur convex, assortativity
trade off forbids at an optimum.

There is exactly one tree on $n=\Delta+1$ vertices with maximum degree
$\Delta$, namely the star $K_{1,\Delta}$, and exactly one on $n=\Delta+2$
vertices. For $n=\Delta+3$ there are three
non-isomorphic trees with maximum degree $\Delta$, of which the one
maximizing $\sigma$ has a single pendant path of length $2$ attached to the
star's center rather than two vertices attached as separate leaves or a path
of length $3$.

\begin{lemma}~\label{lemm21zero}
If $T\in \TT_{n(i,j),\,\Delta}$  contains at least two big vertices, then $m_{1,2}(T)=0$; equivalently, no leaf of $T$ is adjacent to a degree $2$ vertex.
\end{lemma}
\begin{proof}
Suppose toward a contradiction that $uv\in E(T)$ with $d(u)=2,\ d(v)=1$. Since
$T$ has at least two big vertices, let $x$ be the big vertex of $T$ closest to
$u$, and let $y$ be a neighbor of $x$ that is not a leaf,
By choosing the unique path $\mathcal P$ from $y$ to $v$ passes through $x$, as
$\mathcal P = y\,x\cdots u\,v$. Such $y$ exists because $T$ has at least two big
vertices and no big vertex lies strictly between $x$ and $v$ on $\mathcal P$. Thus, the neighbor of $x$ on the $x$ to the other big vertex side is not a leaf. Since $d(v)=1<d(y)$ where $d(y)\geqslant 2$, $y$ not being a leaf. According to Proposition~\ref{prop001pathmonotonicity} by applied with the roles $u\to y,\ v\to v,\ x\to x,$ and $y\to u$ of the path $y\,x\cdots u\,v$ yields $d(x)\leqslant d(u)=2$. Hence, contradicting that $x$ is a big vertex $d(x)\geqslant 3$ and $m_{1,2}(T)=0$.
\end{proof}

\begin{lemma}~\label{lem43twobigvertices}
If $n\ge \Delta+4$, then $T\in \TT_{n(i,j),\,\Delta}$ contains at least two big vertices.
\end{lemma}
\begin{proof}
Suppose toward a contradiction that $T\in \TT_{n(i,j),\,\Delta}$ has at most one big vertex. Since $\Delta\geqslant 3$, $T$ must have exactly one vertex $u$ of degree $\Delta$ and no
other vertex of degree $\ge 3$; thus $T$ consists of exactly $\Delta$ pending
paths attached at $u$.

\medskip

\noindent \textbf{Claim 1.} \textit{No pending path has length $\ge 4$.}  

Suppose one does, assume that $\mathcal P = u\cdots v_3v_2v_1v_0$. Let $T'\in \TT_{n(i,j),\,\Delta}$ where 
$T'=T-v_0v_1+v_0v_2$. This is a tree since $v_0v_1$ is a leaf edge and
$v_0v_2$ reattaches the same leaf one step further along the path; in $T'$,
$d(v_2)$ increases from $2$ to $3$ and $d(v_1)$ decreases from $2$ to $1$,
with all other degrees unchanged. According to Proposition~\ref{prop001pathmonotonicity},  $\sigma(T')-\sigma(T)>0$, contradicting maximality of $T$.

\medskip

\noindent \textbf{Claim 2.} \textit{At most one pending path has length $\geqslant 2$.}  

Suppose two distinct pending paths $v_0v_1\cdots u$ and $u v_2\cdots v_3$ both have length
greater than 2. Applying Proposition~\ref{prop001pathmonotonicity} to the subpath
$v_0\,v_1\cdots u\,v_2$  with $d(v_0)=1<d(v_2)=2$ forces among the contrapositive form of
the second claim of Proposition~\ref{prop001pathmonotonicity}. Thus, we should be established that $d(v_1)\leqslant d(u)$, which is automatic. Hence,  instead we apply the
proposition in the direction that constrains $v_1$ directly. If $d(u)=\Delta$
being the maximum possible degree forces, by the contrapositive if $d(v_1)>d(u)$. Then, $d(v_0)\leqslant d(v_2)$,
which is already true and gives no contradiction directly with the pair $(v_0,v_2)$ playing the role of $(u,v)$ where $d(v_2)>d(v_0)$. Since $d(u)=\Delta\geqslant 3> 2=d(v_1)$ when $\Delta\geqslant 3$, a contradiction. Hence at most one pending path has length
$\geqslant 2$.

\medskip

By proving both \textbf{Claims 1} and \textbf{2}, it follows that,  at most one pending path has length $3$, and all remaining
$\Delta-1$  pending paths have length $1$. Hence, $n\leqslant \Delta+3$, contradicting $n\geqslant \Delta+4$. This completes the proof.
\end{proof}

\begin{corollary}~\label{cor001m12zero}
If $n\ge \Delta+4$, then $m_{1,2}(T)=0$.
\end{corollary}

\begin{lemma}~\label{lem001atmostoneinternalleaf}
If $n\geqslant \Delta+4$, then at most one internal leaf of $T\in \TT_{n(i,j),\,\Delta}$ has degree strictly less than $\Delta$.
\end{lemma}
\begin{proof}
Since $n\geqslant\Delta+4$, $T$ has at least two distinct internal leaves. 
Suppose toward a contradiction that at least two internal leaves, consider $u$ and
$v$, satisfy $d(u),d(v)<\Delta$; and  $u,v$ at maximum possible
tree distance among all such pairs.

\medskip

\noindent \textbf{Case 1.} \emph{If $u,v$ adjacent.} Then $u,v$ are the only non-leaf vertices of $T$, any other non-leaf vertex would be internal and would have to be an
internal leaf adjacent to $u$ or $v$, contradicting that $u,v$ are the unique
maximum distance pair with sub maximal degree unless it too has degree  less than $\Delta$
and is farther, a contradiction to maximality of distance. Since $\max_w d(w)=\Delta$, at least one of $u,v$, in this case, assume that $u$, has degree
$\Delta$. Then $v$ is the only internal leaf of sub maximal degree, proving the claim in this case.

\medskip

\noindent \textbf{Case 2.} \emph{If $d_T(u,v)\geqslant 2$.} Denote the neighbors of $u$ by
$u_1,\dots,u_{d(u)}$, with $u_i$ a leaf for $i\geqslant 2$ and $u_1$ the neighbor on
the $u$ to $v$ path. 

Similarly $v_1,\dots,v_{d(v)}$ for $v$. Since
$n\geqslant \Delta+4$, Lemma~\ref{lem43twobigvertices} gives at least two big
vertices, and Corollary~\ref{cor001m12zero} then gives $m_{1,2}(T)=0$. Thus, every leaf
of $T$ is adjacent to a vertex $\omega$ of degree $d(\omega)\geqslant 3$; in particular $d(u),d(v)\geqslant 3$.

\medskip

\noindent \textbf{Case 3.} \emph{If $d(u)>d(v)$} by applying Proposition~\ref{prop001pathmonotonicity} to the path
$u_1\,u\cdots v\,v_1$   yields $d(u_1)\leqslant d(v_1)$. If $d(u)=d(v)$ and
$d(u_1)>d(v_1)$, relabel $u\leftrightarrow v$ to again obtain $d(u)\geqslant d(v)$ and
$d(u_1)\leqslant d(v_1)$.

Let $T'\in \TT_{n(i,j),\,\Delta}$ where $T' = T - vv_2 + uv_2$. We decompose $\sigma(T')-\sigma(T)=\sigma_1+\sigma_2+\sigma_3$ where:
\begin{itemize}
\item $\sigma_1$ collects the change from edges $uu_2,\dots,uu_{d(u)}$ together with $vv_3,\dots,vv_{d(v)}$ whose common endpoint $v$
decreases from $d(v)$ to $d(v)-1$, $d(v)-2$ such leaf edges, 
\begin{equation}~\label{eqq1sigma1lemma5}
\sigma_1\!=\! (d(v)\!-\!2)\Big[(d(v)\!-\!2)^2\!-\!(d(v)\!-\!1)^2\Big]\!+\!(d(u)\!-\!1)\Big[d(u)^2\!-\!(d(u)\!-\!1)^2\Big].
\end{equation}
\item $\sigma_2$ is the contribution of the moved edge itself, 
\begin{equation}~\label{eqq2sigma1lemma5}
\sigma_2 = d(u)^2 - (d(v)-1)^2.
\end{equation}
\item $\sigma_3$ is the contribution of the two path-edges $uu_1$ and $vv_1$,
\begin{equation}~\label{eqq3sigma1lemma5}
\sigma_3\! =\! \big[(d(u)\!+\!1)\!-\!d(u_1)\big]^2\! -\! \big[d(u)\!-\!d(u_1)\big]^2\!+\! \big[(d(v)\!-\!1)\!-\!d(v_1)\big]^2\!-\! \big[d(v)\!-\!d(v_1)\big]^2.
\end{equation}
\end{itemize}

Thus, by considering $\sigma_1 = (d(u)-d(v)+1)\big(2d(u)+2d(v)-5\big)$ and $d(u)\geqslant d(v)\geqslant 3$, both factors are positive. From \eqref{eqq2sigma1lemma5},
$d(u)\geqslant d(v)>d(v)-1$ gives $\sigma_2 = d(u)^2-(d(v)-1)^2>0$. According to~\eqref{eqq3sigma1lemma5} algebraically,
$\sigma_3 = 2d(u)-2d(v)+2d(v_1)-2d(u_1)+2$, which is $d(v_1)\geqslant d(u_1)$. Hence, according to Proposition~\ref{prop001pathmonotonicity} and based on~\eqref{eqq1sigma1lemma5} satisfying $\sigma(T')-\sigma(T)=\sigma_1+\sigma_2+\sigma_3>0$, contradicting maximality
of $T$.
\end{proof}

A tree with exactly one internal leaf has $n=\Delta+1$ vertices. In this case, for any tree
on $n\ge\Delta+4$ vertices has at least two internal leaves, and
Lemma~\ref{lem001atmostoneinternalleaf} is meaningful for the full range
$n\ge\Delta+4$. On the other hand, a tree on $n\leqslant 2\,\Delta-1$ vertices cannot
carry two vertices of degree $\Delta$ simultaneously. Hence,  the bound ``at most
one'' is best possible in the view $\Delta+4\leqslant n\leqslant 2\,\Delta-1$. Theorem
\ref{thm001internal} below shows that for $n\geqslant 2\,\Delta$ or
$n\geqslant 7$ if $\Delta=3$. This single exceptional low degree internal leaf cannot
occur at all.

\begin{theorem}~\label{thm001internal}
Let $T\in \TT_{n(i,j),\,\Delta}$ be a maximal tree with maximum degree $\Delta$. If $n\geqslant 7$ for $\Delta=3$, or $n\geqslant 2\,\Delta$ for $\Delta\geqslant 4$. Then, every internal leaf of $T$ has degree exactly $\Delta$.
\end{theorem}
\begin{proof}
Suppose toward a contradiction that some
internal leaf $u$ of $T$ has $d(u)<\Delta$. By
Lemma~\ref{lem001atmostoneinternalleaf}, every other internal leaf of $T$
has degree exactly $\Delta$. Since $n\geqslant 2\,\Delta$ and $d(u)<\Delta$, $T$
contains at least one core vertex. Indeed, if $T$ had no core vertex, its non leaf vertices would be exactly its internal leaves, contributing at most
$n-(\text{leaves}) \leqslant \Delta$ vertices with total degree at most $\Delta^2$
edges attaching leaves plus $O(\Delta)$ internal edges, which is insufficient
to reach $n\geqslant 2\Delta$ together with $d(u)<\Delta$.

Let $v$ be any core vertex, and let $w$ be a neighbor of $v$ with $d_T(w,u)>d_T(v,u)$ and $w$ not a leaf such $w$ exists since $v$, being a core vertex, is not itself an internal leaf. Thus, it has at least two non leaf neighbors, at least one of which is farther from
$u$ than $v$ is). If $d(v)>d(u)$, according to Proposition~\ref{prop001pathmonotonicity}
applied to the path from $u$ through $v$ to $w$ where $d(w)\geqslant 2$,
contradicting that $w$ is not a leaf. Hence $d(v)\leqslant d(u)$ for every core
vertex $v$.

We now distinguish two cases according to the maximum number of non-leaf
neighbors of any core vertex.

\medskip

\noindent \textbf{Case 1.} \textit{Every core vertex has exactly two non-leaf neighbors.}

\smallskip

\textbf{Subcase~1.1.} \emph{Every core vertex has degree $2$.} Then $T$ has exactly
two internal leaves, $u$ and some $v$ with $d(v)=\Delta$ by considering 
Lemma~\ref{lem001atmostoneinternalleaf}, joined by a path $\mathcal P$ whose
interior vertices are precisely the core vertices, all of degree $2$.

If $\mathcal P$ has exactly one interior core vertex $w$. Then, $n=1+(d(u)-1)+1+1+(\Delta-1)+\cdots$.  Thus, by counting all vertices of $T$ gives
$n=d(u)+1+\Delta$. Since $n\geqslant 2\,\Delta$ and $d(w)=2$ one
obtains $d(u)=\Delta-1$.  Let $T'\in \TT_{n(i,j),\,\Delta}$ where  $T' = T-vw+vu$. Then, the edges $vw,uv$ contribute
$0-(\Delta-2)^2$ to $\sigma(T')-\sigma(T)$ where edge $vw$, contribution $0$ since
$d(v)=d(w)=\ldots$, the edge $uw$   contributes $(\Delta-1)^2-(\Delta-3)^2$, and the $\Delta-2$ leaves attached to $u$ contribute $(\Delta-1)^2-(\Delta-2)^2$ each. Therefore,
\begin{equation}~\label{eqq01subcase1a}
\sigma(T')-\sigma(T)=\Delta^2+\Delta-6,
\end{equation}
which is strictly positive for every $\Delta\geqslant 3$. Indeed
$\Delta^2+\Delta-6=(\Delta+3)(\Delta-2)>0$ for $\Delta>2$, a contradiction.

\medskip 

If $\mathcal P$ has more than one interior vertex, let $\mathcal P'=zwu$ be the subpath nearest $u$ where $w,z$ are core vertices of degree $2$, or $z=v$ if $\mathcal P$ has exactly two interior vertices. Let $T'\in \TT_{n(i,j),\,\Delta}$ where $T'=T-zw+zu$. The edges $zw,zu$ contribute $(d(u)-2)^2$, the edge $wu$ contributes $(d(u)-1)^2-(d(u)-2)^2$, and the $d(u)-1$ leaves attached to $u$ contribute $d(u)^2-(d(u)-1)^2$ each. Thus, 
\begin{equation}~\label{eqq01subcase1amulticore}
\sigma(T')-\sigma(T)\! =\! (d(u)\!-\!2)^2\!+\!(d(u)\!-\!1)^2\!-\!(d(u)\!-\!2)^2\!+\!(d(u)\!-\!1)\big[d(u)^2\!-\!(d(u)\!-\!1)^2\big], 
\end{equation}
which is holds $\sigma(T')-\sigma(T)=3d(u)^2-5d(u)+2,$ strictly positive for every $d(u)\geqslant 2$, it factors as $(3d(u)-2)(d(u)-1)$, both factors positive for $d(u)\geqslant 2$, a contradiction.

\textbf{Subcase~1.2.} \emph{ Some core vertex $v$ has degree $d(v)\geqslant 3$.} Assume that  $v$ to be the degree $d(v)\geqslant 3$ core vertex closest to $u$. Thus, every interior vertex of  $\mathcal P$ from $u$ to $v$ has degree $2$. Since $v$ has exactly two non leaf
neighbors   and $d(v)\geqslant 3$, $v$ has at least one pendant leaf $v_1$. Let $T'\in \TT_{n(i,j),\,\Delta}$ where $T'=T-v_1v+v_1u$. Since $v$ is a core vertex, $d(v)\leqslant d(u)$, the pair $v_1v,v_1u$ contributes strictly positively to $\sigma(T')-\sigma(T)$, and the edges of $\mathcal P$ likewise contribute a strictly positive sum where each edge of $\mathcal P$ sees one endpoint's ``target'' degree effectively shift from $d(v)$ toward
$d(u)$, and $d(v)\leqslant d(u)$ makes each such shift favorable.

Let $v_2$ be the non leaf neighbor of $v$ not on $\mathcal P$, and $u_1$ a leaf
attached to $u$. If $d(v_2)\geqslant d(v)$, the edge $v_2v$ contributes positively;
otherwise its contribution is bounded below by $(d(v)-2)^2-(d(v)-1)^2$. Hence, $\sigma_1$ of $v_2v$ and $u_1u$ satisfies
\begin{equation}~\label{eqq01subcase1bsigma1}
\sigma_1 \geqslant (d(v)-2)^2-(d(v)-1)^2 + d(u)^2-(d(u)-1)^2. 
\end{equation}
Eq.~\eqref{eqq01subcase1bsigma1} emphasizes that $\sigma_1 \geqslant 2d(u)-2d(v)+2 > 0$. Let $\mathcal L_v$ (resp.\ $\mathcal L_u$) be the leaf edges at $v$
(resp.\ $u$) other than $v_1v$ (resp.\ $u_1u$). Since $d(v)\leqslant d(u)$ and $v$
has one fewer available leaf slot than $u$ when $v_2v$ occupies a non-leaf
slot at $v$), $|\mathcal L_v|<|\mathcal L_u|$. According to~\eqref{eqq01subcase1amulticore},~\eqref{eqq01subcase1bsigma1} and
the positive contribution of $\mathcal P$'s edges,
\begin{equation}~\label{eqq01subcase1bfinal}
\sigma(T')-\sigma(T) > |\mathcal L_v|\Big[(d(v)-2)^2-(d(v)-1)^2\Big]
- |\mathcal L_u|\Big[d(u)^2-(d(u)-1)^2\Big]
> |\mathcal L_v|\big(2d(u)-2d(v)+2\big),
\end{equation}
where $\sigma(T')-\sigma(T) >0$ a contradiction.

\medskip

\noindent \textbf{Case 2.} \textit{Some core vertex $v$ has at least three non leaf neighbors.} Assume that $v$ to be such a vertex closest to $u$, and let $\mathcal P$ be the path from $u$ to $v$, all of whose interior vertices have degree $2$.

\smallskip

\textbf{Subcase~2.1.} \emph{Every non leaf neighbor of $v$ other than possibly the one
on $\mathcal P$ is an internal leaf.}  Then $v$ is the unique core vertex of $T$ and $u$
is adjacent to $v$. If $v$ has a pendant leaf $w$, let $T'\in \TT_{n(i,j),\,\Delta}$ where $T'=T-wv+wu$; since
$d(v)\leqslant d(u)$, the pair $wv,wu$ and the edge $uv$ both contribute positively,
each internal leaf neighbor $z$ of $v$   contributes positively via $vz$,
and only the leaf edges $\mathcal L_v$ at $v$  can contribute
negatively, each by at least $(d(v)-2)^2-(d(v)-1)^2$; since
$|\mathcal L_v|<|\mathcal L_u|$, the same telescoping as
\eqref{eqq01subcase1bfinal} gives $\sigma(T')-\sigma(T)>0$, a contradiction.
If instead $v$ has no pendant leaf, let $z$ be any leaf of $T$ not attached to
$u$ and consider $T'=T-uv+uz$. Then, 
\begin{equation}~\label{eqq01subcase2anoleaf}
\sigma(T')-\sigma(T) >
(d(v)-1)\Big[(\Delta\!-\!(d(v)\!-\!1))^2\!-\!(\Delta\!-\!d(v))^2\Big]
\!+\!(\Delta\!-\!2)^2\!-\!(\Delta\!-\!1)^2.
\end{equation}
Thus, $\sigma(T')-\sigma(T) >(d(v)-2)(2\Delta-2d(v)-1) > 0,$
since $3\leqslant d(v)<\Delta$ when $v$ is a core vertex with $d(v)\leqslant d(u)<\Delta$,
 by using $d(u)<\Delta$, a contradiction.

\smallskip

\textbf{Subcase~2.2.} \emph{If $v$ has a non leaf neighbor that is itself a core vertex.} 
Let $u_1,v_1$ be the neighbors of $u,v$ on $P$. If $v_1$ is the only
core vertex neighbor of $v$, then $v$'s remaining non-leaf neighbor $v_2$ is
an internal leaf, so $d(v_2)=\Delta$ according to
Lemma~\ref{lem001atmostoneinternalleaf}. Since $d(u)<\Delta=d(v_2)$,
Proposition~\ref{prop001pathmonotonicity} applied to the path from $u$ through
$u_1$ to $v_2$   forces $d(u_1)\geqslant d(v)$, contradicting
$d(u_1)\leqslant d(v)$.

Hence $v$ has a core vertex neighbor $v_2\ne v_1$. Then, assume that of minimum degree
among all such neighbors where $d(v_2)\leqslant d(u)$.
Let $v_3\ne v_1$ be the third non leaf neighbor of $v$ where minimality of $v_2$
gives $d(v_2)\le d(v_3)$. Let $T'=T-v_2v+v_2u$ and the edges of $\mathcal P$ contribute a
strictly positive sum  as $d(v)\leqslant d(u)$). In this case, let $u_2,u_3$ be two leaves at $u$, and let $\sigma_1$ be the combined contribution of $v_2v, v_2u, v_3v, u_2u,$ and $u_3u$ gives
\begin{equation}~\label{eqq001subcase2bsigma1}
\sigma_1 = \big(d(u)-d(v_2)\big)^2+3d(u)+d(v)\big(d(v)-2\big)
+3\big(d(u)-d(v_2)\big)+2d(v_3),
\end{equation}
where $\sigma_1>0$ and by using $d(u)\geqslant d(v_2)$. With $\mathcal L_v,\mathcal L_u$ the remaining leaf edges at $v,u$
respectively ($|\mathcal L_v|=d(v)-3\leqslant d(u)-3=|\mathcal L_u|$ since $d(v)\leqslant d(u)$), which \eqref{eqq01subcase1bfinal} yields
$\sigma(T')-\sigma(T)>0$, a contradiction.

In every case we reach a contradiction, so no internal leaf of $T$ can have
degree less than $\Delta$.
\end{proof}

\subsection{Bounds on Degree 2 Edges}
Proposition~\ref{propm22boundgeneral001} is not yet the sharpest possible bound; we will show   that in fact $m_{2,2}(T)\leqslant 1$ for every $\Delta\geqslant 3$ once $n$ is large enough, sharpening this bound by one.

\begin{proposition}~\label{propm22boundgeneral001}
If $n\ge\Delta+4$, then $m_{2,2}(T)\leqslant 2$.
\end{proposition}
\begin{proof}
Suppose toward a contradiction that $m_{2,2}(T)\ge3$ and let $e_i=u_iv_i$,
$i\in\{1,2,3\}$, be three edges with $d(u_i)=d(v_i)=2$.

\noindent \textbf{Case 1.} \emph{All three vertex-disjoint.} Let $w_i$ be the neighbor of $v_i$ other than $u_i$. Let $T, T'\in \TT_{n(i,j),\,\Delta}$ where 
\begin{equation}~\label{eqq01m22transformdisjoint}
T'=T - \mathcal P_3,
\end{equation}
 and $\mathcal P_3=u_3v_3-v_3w_3+u_3w_3-u_2v_2-v_2w_2+u_2w_2-v_1w_1+v_1v_2+v_2w_1+v_3v_1.$ The contributions of $\{v_iw_i,u_iw_i\}$ cancel for $i\in\{1,2\}$ each pair
of edges is replaced by an edge of the same two endpoint-degree values when $d(u_i)=d(v_i)=2$ throughout. The contribution of each $u_iv_i$ is $0$, and the contribution of $\{v_1w_1,v_2w_1\}$
cancels similarly. What remains is the pair of new edges $v_2w_1$ was already
counted, and according to Proposition~\ref{prop001pathmonotonicity} and \eqref{eqq01m22transformdisjoint} we obtain $\sigma(T')-\sigma(T)>0$, a contradiction.

\noindent \textbf{Case 2.} \emph{Two share an endpoint $v_1=u_2$, third disjoint.} Let
$T'=T-u_3v_3-v_3w_3+u_3w_3+v_3v_1$. The pair $\{v_3w_3,u_3w_3\}$ cancels, $u_3v_3$ contributes $0$, each of $u_1v_1,u_2v_2$ contributes $(3-2)^2-0=1$, giving $\sigma(T')-\sigma(T)>0$,
a contradiction.

\noindent \textbf{Case 3.}  \emph{Two pairs share endpoints $v_1=u_2$ and $v_2=u_3$.} With the same case in \textbf{Case 2}, 
$T'=T-u_3v_3-v_3w_3+u_3w_3+v_3v_1$, an identical computation where $v_1$ has degree $3$ from both the $u_1v_1$ side and gains a further edge $v_3v_1$, and $v_2$ similarly gives $\sigma(T')-\sigma(T)>0$, again a contradiction.

 Therefore, according to cases \textbf{Case 1, 2} and \textbf{3}, in every configuration we reach a contradiction. Then, $m_{2,2}(T)\leqslant 2$.
\end{proof}

Let  $T\in \TT_{n(i,j),\,\Delta}$ is a maximal tree with maximum degree $\Delta\geqslant 3$ and $n\geqslant n_0(\Delta)$ for a threshold to be determined; by
Theorem~\ref{thm001internal} we assume $n\geqslant 2\Delta$
and every internal leaf of $T$ has degree
$\Delta$. Our goal is to eliminate every vertex degree strictly between $2$ and $\Delta$.

The natural strategy for general $\Delta$ is an induction on $d$, for $d=3,4,\dots,\Delta-1$ by assuming $n_3=n_4=\cdots=n_{d-1}=0$ has already been established, we should be show that $n_d=0$. Then, fix $d$ with $3\leqslant d\leqslant \Delta-1$, and suppose $T$ contains an edge $uv$ with $d(u)=d(v)=d$. Assume that $k=\Delta-d\ge1$. The transformation used in \cite[Lemma~8]{DimitrovPaper2025} for $(d,\Delta)=(3,5)$ where $k=2$ relocates all of $u$'s neighbors other than $v$ onto $v$; since $u$ has $d-1$
such neighbors.  We generalize it as follows: relocate $\min(k,d-1)$ of $u$'s other neighbors onto $v$. We restrict attention here to the regime $k\leqslant d-1$, i.e.\ $d\geqslant \lceil(\Delta+1)/2\rceil$, in which exactly $k$ neighbors can be relocated so that $v$'s degree becomes exactly $\Delta$; where $m:=2d-\Delta\geqslant 1$ for $u$'s resulting degree.

\begin{lemma}~\label{lem001attempted8}
Let $u_1,\dots,u_k$ be the relocated neighbors of $u$, and let $T'\in \TT_{n(i,j),\,\Delta}$ be the
tree obtained from $T\in \TT_{n(i,j),\,\Delta}$ by deleting edges $u_iu$ where $i=1,\dots,k$ and adding
edges $u_iv$. Then
\begin{equation}~\label{eqq01attempted8bound}
\sigma(T')-\sigma(T) \;\geqslant\; k^2(5-\Delta) \;+\; (m-1)\,k\,(\Delta-3d+2).
\end{equation}
\end{lemma}

\begin{proof}
Assume that $\Delta-1$ edges incident to $v$, each seeing
its non $v$ endpoint's target degree change from $d$ to $\Delta$. Then,  each contributes $(x-\Delta)^2-(x-d)^2$ for some $x\in[1,\Delta]$ satisfying  
\begin{equation}~\label{eqq01groupA}
(x-\Delta)^2-(x-d)^2 = (\Delta-d)(\Delta+d-2x) = k(\Delta+d-2x),
\end{equation}
which is minimized over $x\in[1,\Delta]$ at $x=\Delta$, Eq.~\eqref{eqq01groupA} giving $k(\Delta+d-2\Delta)=k(d-\Delta)=-k^2$. Thus, by summing  over all
$\Delta-1$ such edges gives $-(\Delta-1)k^2$.

If $m-1$ edges remaining at $u$, each seeing $u$'s target degree
change from $d$ to $m$. Then,  each contributes $(x-m)^2-(x-d)^2$, and  since $m-d=-k$ according to~\eqref{eqq01groupA},
\begin{equation}~\label{eqq02groupB}
(x-m)^2-(x-d)^2 = -k(m+d-2x) = k(2x-m-d),
\end{equation}
minimized at $x=1$, giving $k(2-m-d)$. Thus, by summing over $m-1$ such edges gives $(m-1)k(2-m-d)$.  From~\eqref{eqq02groupB} the edge $uv$ contributes $(m-\Delta)^2-(d-d)^2=4k^2$ exactly, since $m-\Delta=-2k$. Then,  $-(\Delta-1)k^2+4k^2 = k^2(5-\Delta)$. Assume that $m=2d-\Delta$, according to~\eqref{eqq02groupB}, $2-m-d=2-(2d-\Delta)-d=\Delta-3d+2$,
gives the stated form \eqref{eqq01attempted8bound}.
\end{proof}

Actually, the coefficient $k^2(5-\Delta)$ is strictly negative for every $\Delta>5$, regardless of $d$. For example, at $\Delta=6$, $d=\lceil7/2\rceil=4$, the bound
\eqref{eqq01attempted8bound} evaluates to
$4(5-6)+1\cdot2\cdot(6-12+2)=-12<0$; at $\Delta=9$, $d=5$ ($k=4,m=1$):
$16(5-9)+0=-64<0$. In every case we checked with $\Delta>5$, the worst case
bound is negative. That is meaning the direct generalization of the
transformation in \cite[Lemma~8]{DimitrovPaper2025} does not, by itself,
establish a contradiction for $\Delta\ne5$. An adversarial choice of the neighbor degrees in~\eqref{eqq01groupA} and \eqref{eqq02groupB} can, as far as this bound shows, make
$\sigma(T')-\sigma(T)<0$. Thus, maximality of $T$ gives no contradiction from this transformation alone.

This is a genuine mathematical obstruction, not merely a matter of sharper bookkeeping where the coefficient $(5-\Delta)$ multiplying $k^2$ in \eqref{eqq01attempted8bound} arises because~\eqref{eqq01groupA} has exactly $\Delta-1$
terms. We record this as the central open problem of the general $\Delta$ characterization.

\begin{conjecture}~\label{conj002mddzero}
For every $\Delta\ge3$ and every $d$ with $3\le d\le\Delta-1$, there exists $n_1(\Delta,d)$ such that every maximal tree $T\in \TT_{n(i,j),\,\Delta}$ on $n\geqslant n_1(\Delta,d)$
vertices with maximum degree $\Delta$ satisfies $m_{d,d}(T)=0$.
\end{conjecture}

Conjecture~\ref{conj002mddzero} is supported by three independent lines of evidence, none of which by itself constitutes a proof. 
\begin{enumerate}
    \item It is a theorem for $(\Delta,d)=(4,3)$ and $(\Delta,d)=(5,3),(5,4)$ (see~\cite{VukicevicEtAl2024,DimitrovPaper2025}).
\item  The asymptotic hub connector family comparison of
Section~\ref{sec03framework} shows that degree $2$ connectors strictly dominate degree $d$ connectors, $3\leqslant d\leqslant \Delta-1$, in $\sigma$ density as $n\to\infty$, for every
$\Delta\ge4$.

\item  Conjecture~\ref{conj002mddzero} for $\Delta\ne 5$ requires a
 qualitatively different transformation, plausibly one that
relocates neighbors from  both  $u$ and $v$ symmetrically toward two distinct  existing degree $\Delta$ hubs elsewhere in the tree.
\end{enumerate}

%\subsection{Bounding $m_{2,2}$ and $m_{\Delta,\Delta}$}~\label{subchart001}

\subsection{%
  Bounding \texorpdfstring{$m_{2,2}$}{m(2,2)}
  and \texorpdfstring{$m_{\Delta,\Delta}$}{m(Delta,Delta)}%
}\label{subchart001}

In contrast to Proposition~\ref{propm22boundgeneral001} the bounds on $m_{2,2}$ and
$m_{\Delta,\Delta}$ generalize with only mild restrictions, because their
proofs involve only the two  surviving degrees $2$ and $\Delta$, not the problematic intermediate range. Both results below are conditional on property $\mathcal P_1^\Delta$. 

\begin{proposition}~\label{lem001m22sharpgeneral}
Let $T\in \TT_{n(i,j),\,\Delta}$ be a maximal tree with maximum degree $\Delta$, $3\leqslant\Delta\leqslant 8$,
$n\geqslant\Delta+4$, satisfying $\mathcal P_1^\Delta$. Then $m_{2,2}(T)\leqslant 1$.
\end{proposition}
\begin{proof}
Let $z$ be an internal leaf (degree $\Delta$ by
Theorem~\ref{thm001internal}) with leaf-neighbor $z_1$, and
let $T'$ be obtained by deleting all edges at $u,v,a,b$ and adding
$u_1v_1,\,a_1b_1,\,uz_1,\,vz_1,\,az_1,\,bz_1$. Each of
$u,v,a,b$ becomes a leaf attached to $z_1$, which rises from degree $1$ to degree $\Delta$. Then,  each of $u_1,v_1,a_1,b_1$ retains degree $\Delta$ and  the edge
$zz_1$ changes from contributing $(\Delta-1)^2$ to contributing $0$),
\begin{equation}~\label{eqq001m22sharpdiff}
\sigma(T')-\sigma(T)= 4(\Delta-1)^2 - 4(\Delta-2)^2 - (\Delta-1)^2.
\end{equation}
Thus,  for every integer $\Delta$ with $2\leqslant\Delta\leqslant 8$, giving $\sigma(T')-\sigma(T),$ a contradiction with the maximality of $T$ throughout this range.
\end{proof}

For $\Delta\ge9$, \eqref{eqq001m22sharpdiff} is non positive (e.g.\
$-4$ at $\Delta=9$), so this specific transformation fails to certify $m_{2,2}\leqslant 1$.

\begin{proposition}~\label{prop001mDeltaDeltageneral}
Let $T\in \TT_{n(i,j),\,\Delta}$ be a maximal tree with maximum degree $\Delta\geqslant 5$, satisfying
$\mathcal P_1^\Delta$ and $\mathcal P_2^\Delta$, with $n$ sufficiently large that an internal leaf $u$ and $\Delta-1$ pairwise vertex-disjoint edges $a_1b_1,\dots,a_{\Delta-1}b_{\Delta-1}$ with $d(a_i)=d(b_i)=\Delta$, disjoint from the edges at $u$, where $m_{\Delta,\Delta}(T)\ge\Delta$. Then, $m_{\Delta,\Delta}(T)\leqslant \Delta-1$.
\end{proposition}

\begin{proof}
Let $u$ be an internal leaf with leaf neighbors $u_1,\dots,u_{\Delta-1}$ and unique non leaf neighbor $p$. Suppose
$m_{\Delta,\Delta}(T)\geqslant\Delta$, and assume that the $\Delta-1$ disjoint edges $a_ib_i$ as above. Let
$\mathcal L^-=\{u_iu: i=1,\dots,\Delta-1\}\cup\{a_ib_i:i=1,\dots,\Delta-1\}$ and
$\mathcal L^+=\{a_iu_i,\,u_ib_i:i=1,\dots,\Delta-1\}$, and consider $T'=T-\mathcal L^-+\mathcal L^+$. Then, each $u_i$ rises from degree $1$ to degree $2$ had attached to $a_i,b_i$, each of
degree $\Delta$ and $u$ falls from degree $\Delta$ to degree $1$ where retaining only the edge $up$. Since $(1-d(p))^2-(\Delta-d(p))^2=(\Delta-1)\bigl(2d(p)-\Delta-1\bigr)$, 
\begin{equation}~\label{eqqtest001}
\sigma(T')-\sigma(T)=(\Delta-1)\left[\Delta^2-7\Delta+6+2d(p)\right].
\end{equation}
Since $p$ is a non leaf vertex, $d(p)\geqslant 2$ under $\mathcal P_1^\Delta$, so this term is minimized
 at $d(p)=2$, giving $(\Delta-1)(3-\Delta)$. Thus, by substituting into \eqref{eqqtest001},
\begin{equation}~\label{eqqtest002}
\sigma(T')-\sigma(T)\geqslant (\Delta-1)(\Delta-2)(\Delta-5),
\end{equation}
which is strictly $\sigma(T')-\sigma(T)>0$ for every $\Delta>5$. Thus, $m_{\Delta,\Delta}(T)\leqslant \Delta-1$.
\end{proof}

For $\Delta\ge3$, define:
\begin{align*}
\mathcal P_1^\Delta :&\ T\in \TT_{n(i,j),\,\Delta} \text{ contains only vertices of degree } 1,\,2,\,\Delta;\\
\mathcal P_2^\Delta :&\ \text{all internal leaves of } T \text{ have degree } \Delta;\\
\mathcal P_3^\Delta :&\ m_{\Delta,\Delta}(T)\le\Delta-1,  m_{2,2}(T)\leqslant 2;\\
\mathcal P_4^\Delta:&\ \text{if } m_{2,2}(T)=1, \text{ then } m_{\Delta,\Delta}(T)=0.
\end{align*}

\begin{theorem}~\label{thm03partialgeneral}
Let $T\in \TT_{n(i,j),\,\Delta}$ be a maximal tree with maximum degree $\Delta\geqslant 3$ and
$n\geqslant \max(2\Delta,\,n_1(\Delta,3),\dots,n_1(\Delta,\Delta-1))$ with $n_1(\Delta,d)$ as in Conjecture~\ref{conj002mddzero}. Then,
\begin{enumerate}
\item  $T$ satisfies $(P_2^\Delta)$ unconditionally, for every
$\Delta\geqslant 3$;
\item  $T$ satisfies $\mathcal P_1^\Delta$ if Conjecture~\ref{conj002mddzero}
holds for $\Delta \in \{4,5\}$; 
\item  conditional on (2), $T$ satisfies the coarse form of
$\mathcal P_3^\Delta$ where $m_{\Delta,\Delta}\leqslant \Delta-1$ for every $\Delta\geqslant 5$, and the sharpened form
$m_{2,2}\leqslant 1$ for $3\leqslant \Delta\leqslant 8$, and $m_{2,2}\leqslant 2$ unconditionally for every $\Delta\geqslant 3$.
\end{enumerate}
\end{theorem}
\begin{proof}
Immediate from Theorem~\ref{thm001internal},
Conjecture~\ref{conj002mddzero}, Proposition~\ref{propm22boundgeneral001},
Proposition~\ref{lem001m22sharpgeneral}, and
Proposition~\ref{prop001mDeltaDeltageneral}, as cited.
\end{proof}

%%%%%%%%%%%%%%%%%%%%%%%%%%%%%%%%%%%%%%%%%%%%%%%%%%%%%%%%%%%%%%%%%%%%%%%%%%%%%%%%%%%%%%%%%%%%%%%%%%%%%%%%%%%%%%%%%%%%%%%%%%%%%%%%%%%%%%%%%%%%%%%%%

\section{Closed-Form Extremal \texorpdfstring{$\sigma$}-Irregularity and Threshold Analysis}~\label{sec04closedform}

Throughout this section we work \emph{conditionally} on
Conjecture~\ref{conj002mddzero}, i.e., we assume $T$ is a maximal tree satisfying $\mathcal P_1^\Delta$ where  every vertex of $T$ has degree $1$, $2$, or $\Delta$. As established in subsection~\ref{subchart001}, this is unconditionally justified for $\Delta\in\{3,4,5\}$.

\subsection{A Reduced-Tree Identity}
By Lemma~\ref{lemreducedtreedegrees501} had established a tree on $n_1+n_\Delta$ vertices.

\begin{lemma}~\label{lemreducedtreedegrees501}
Let $T^\bullet$ is a tree on $n_1+n_\Delta$ vertices in which every leaf has degree $1$ and every non leaf vertex has degree exactly $\Delta$.
\end{lemma}
\begin{proof}
Since $T^\bullet$ is a tree on $n_1+n_\Delta$ vertices. Then, suppression of a degree $2$ vertex preserves the tree property and does not change the degree of any other vertex, since it merely merges the two edges at a degree $2$ vertex into one edge between its two neighbors. Since $m_{1,2}(T)=0$, no leaf of $T$ lies on a suppressible path. Then,  all $n_1$ leaves survive unchanged in $T^\bullet$; since $\mathcal P_1^\Delta$ permits only degrees $1,2,\Delta$, every non suppressed, non leaf vertex has degree $\Delta$, and there are $n_\Delta$ of $T^\bullet$.
\end{proof}

\begin{proposition}~\label{prop001reducedtreeidentity}
$n_1 = (\Delta-2)\,n_\Delta + 2$.
\end{proposition}
\begin{proof}
By Lemma~\ref{lemreducedtreedegrees501}, $T^\bullet$ is a tree on
$n_1+n_\Delta$ vertices. Hence, it has $n_1+n_\Delta-1$ edges. Then, $n_1\cdot1 + n_\Delta\cdot\Delta = 2(n_1+n_\Delta-1).$ Since $\Delta n_\Delta - 2n_\Delta = n_1-2$, gives
$n_1=(\Delta-2)n_\Delta+2$.
\end{proof}

\begin{corollary}~\label{cor002n2formula}
$n_2=n - 2 - (\Delta-1)n_\Delta$.
\end{corollary}
\begin{proof}
According to Proposition~\ref{prop001reducedtreeidentity},  $n = n_1+n_2+n_\Delta =[(\Delta-2)n_\Delta+2]+n_2+n_\Delta$, giving
$n_2=n-2-(\Delta-2)n_\Delta-n_\Delta=n-2-(\Delta-1)n_\Delta$.
\end{proof}

\begin{proposition}~\label{propsigmaformula502}
If $T\in \TT_{n(i,j),\,\Delta}$ satisfies $\mathcal P_1^\Delta$ and $m_{1,2}(T)=0$, then
\begin{equation}~\label{eqq001sigmaunderP1}
\sigma(T) \;=\; n_1(\Delta-1)^2 \;+\; 2\big(n_2-m_{2,2}(T)\big)(\Delta-2)^2.
\end{equation}
\end{proposition}
\begin{proof}
Under $\mathcal P_1^\Delta$, every edge of $T$ has both endpoints in $\{1,2,\Delta\}$. In this case, edges of type $(1,1)$ are impossible in a tree with $n\geqslant 3$ and edges of type $(1,2)$ do not occur, since $m_{1,2}(T)=0$. Hence,  every one of the $n_1$
leaf edges is of type $(1,\Delta)$, contributing $(\Delta-1)^2$ each, for a total of $n_1(\Delta-1)^2$ and  any edges of type $(2,2)$ contribute $0$ each. Thus, any edges of type $(\Delta,\Delta)$ contribute $0$ each and any edges of type
$(2,\Delta)$ contribute $(\Delta-2)^2$ each where counting them via the degree sum at degree $2$ vertices, $2n_2 = 2\,m_{2,2}(T) + m_{2,\Delta}(T)$, each
degree $2$ vertex contributes $2$ to the left side, split between $(2,2)$ edges, counted twice, and $(2,\Delta)$ edges, counted once. Thus, 
$m_{2,\Delta}(T)=2\big(n_2-m_{2,2}(T)\big)$ which \eqref{eqq001sigmaunderP1} holds.
\end{proof}

The sharp bound
on $m_{2,2}$ among Proposition~\ref{lem001m22sharpgeneral}, which limits how many degree $2$ vertices can be packed onto a fixed number of hub to hub connections as show that in Figure~\ref{fig000hub}. Thus, among Lemma~\ref{lembufferpacking505} we had established new tree $T^{\bullet\bullet}$ obtained from $T^\bullet$ with certain terms.

\begin{lemma}~\label{lembufferpacking505}
Let $T^{\bullet\bullet}$ be the tree obtained from $T^\bullet$ by deleting its $n_1$ leaves, the induced subtree on the $n_\Delta$ hub vertices alone, it has $n_\Delta-1$ edges. If $m_{2,2}(T)\leqslant 1$, then $n_2 \;\leqslant\; n_\Delta.$
\end{lemma}
\begin{proof}
Since $T^{\bullet\bullet}$ is the tree induced on the $n_\Delta$ hub vertices, it has exactly $n_\Delta-1$ edges, and each such edge
$e\in E(T^{\bullet\bullet})$ corresponds bijectively to a maximal path $\mathcal P_e$ in $T$ joining two hubs through a sequence of
degree $2$ vertices. Let $t(e)\geqslant 0$ for the number of interior degree $2$ vertices of $\mathcal P_e$. Then, 
\begin{equation}~\label{eqq002n2sum}
n_2 \;=\; \sum_{e\in E(T^{\bullet\bullet})} t(e),
\end{equation}
since every degree $2$ vertex of $T$ lies in the interior of exactly one such path.

\smallskip
\noindent\textbf{Claim 1.} \emph{For each $e\in E(T^{\bullet\bullet})$,
the path $\mathcal P_e$ contributes exactly $\max(t(e)-1,\,0)$ edges to $m_{2,2}(T)$, and every $(2,2)$ edge of $T$ arises this way from a unique $e$.}

If $t(e)=0$, $\mathcal P_e$ is a single hub hub edge of $T$, contributing no $(2,2)$ edge. If $t(e)\geqslant 1$, assume that the interior vertices of $\mathcal P_e$ in order
as $w_1,\dots,w_{t(e)}$; the edges of $\mathcal P_e$ are then
$hw_1,\,w_1w_2,\,\dots,\,w_{t(e)-1}w_{t(e)},\,w_{t(e)}h'$, where $h,h'$ are
the two hubs. The edges $hw_1$ and $w_{t(e)}h'$ have one endpoint of degree
$\Delta$ and are therefore not of type $(2,2)$ and each of the remaining $t(e)-1$ edges $w_iw_{i+1}$ joins two degree $2$ vertices and is of type
$(2,2)$. Hence $\mathcal P_e$ contributes exactly $t(e)-1$ edges to $m_{2,2}(T)$ when
$t(e)\geqslant 1$, and $0=\max(0-1,0)$ when $t(e)=0$. Thus, in both cases this equals
$\max(t(e)-1,0)$. Since every degree $2$ vertex of $T$ lies on a unique path $\mathcal P_e$ where the paths $\{\mathcal P_e\}_{e\in E(T^{\bullet\bullet})}$ partition
$V(T)\setminus\{\text{hubs},\text{leaves}\}$ together with the hub
endpoints, and  every $(2,2)$ edge of $T$ lies in the interior of a unique such
path, proving the second assertion. Thus, over all $e$,
\begin{equation}~\label{eqq003m22sum}
m_{2,2}(T) \;=\; \sum_{e\in E(T^{\bullet\bullet})} \max\big(t(e)-1,\,0\big).
\end{equation}

\smallskip
\noindent\textbf{Claim 2.} \emph{If $m_{2,2}(T)\leqslant 1$, then at most one edge of $T^{\bullet\bullet}$ satisfies $t(e)\geqslant 2$, and for that edge $t(e)\leqslant 2$.}

Since $m_{2,2}(T)>0$, according to~\eqref{eqq003m22sum} and a summand
is positive precisely when $t(e)\geqslant 2$, in which case it equals $t(e)-1\geqslant 1$. If two distinct edges $e_1,e_2\in E(T^{\bullet\bullet})$ both satisfied
$t(e_i)\geqslant 2$, their summands would each contribute at least $1$ to
\eqref{eqq003m22sum}, giving $m_{2,2}(T)\geqslant 2$, contrary to hypothesis. Hence at most one edge has $t(e)\geqslant 2$. If such an edge $e^\ast$ exists, its summand
$t(e^\ast)-1$ is itself bounded above by the total $m_{2,2}(T)\le1$. Thus, $t(e^\ast)\leqslant 2$.

\smallskip
  By Claim~2, every edge of
$T^{\bullet\bullet}$ other than the possible exceptional edge $e^\ast$ satisfies $t(e)\leqslant 1$, while $t(e^\ast)\leqslant 2$ when $e^\ast$ exists. Thus, by applying
these bounds term by term in \eqref{eqq003m22sum}, and using
$|E(T^{\bullet\bullet})|=n_\Delta-1$,
\begin{equation}~\label{taptabeqq01}
n_2 \;=\; \sum_{e\in E(T^{\bullet\bullet})} t(e)
\;\leqslant\; n_\Delta,
\end{equation}
where the right hand of inequality~\eqref{taptabeqq01} holds whether or not $e^\ast$ exists. 
\end{proof}

\begin{figure}[H]
    \centering
    \includegraphics[width=0.7\linewidth]{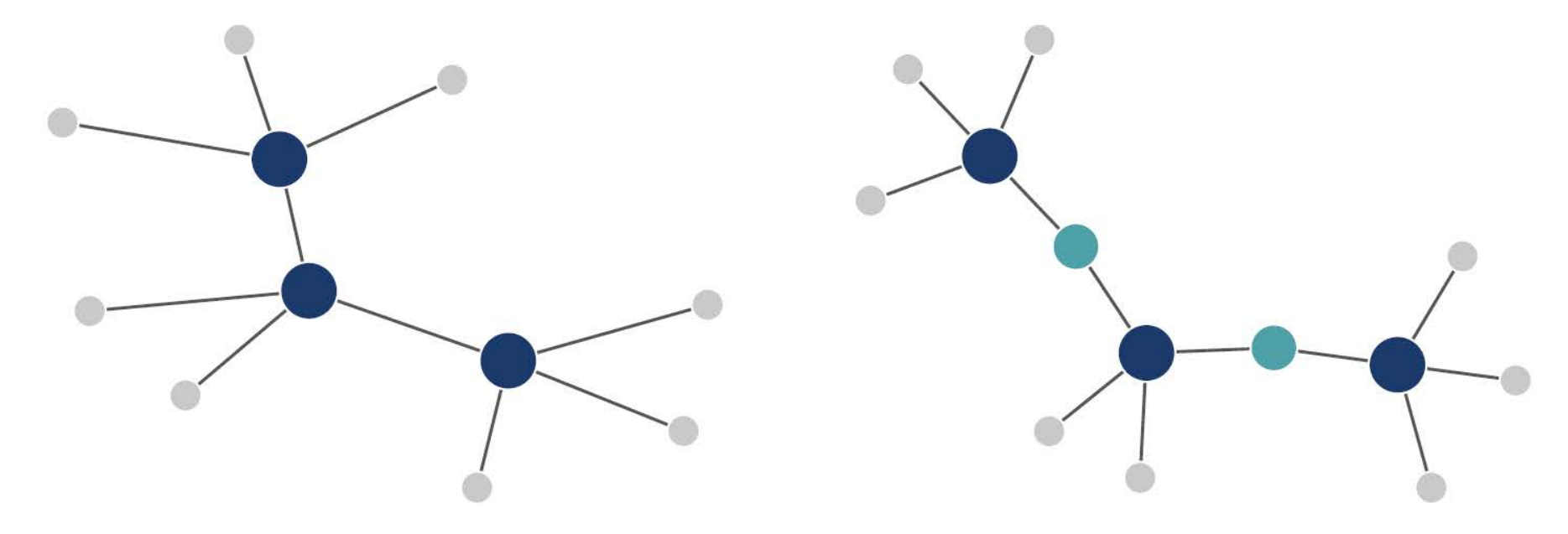}
    \caption{The left side is $\mathcal C_1(h{=}3,\Delta{=}4)$, direct hub hub joins, $n=11$ and the right side is $\mathcal C_2(h{=}3,\Delta{=}4)$, degree 2 buffers, $n=13$.}
    \label{fig000hub}
\end{figure}
\begin{theorem}~\label{thm04optimalnDelta}
Let $T\in \TT_{n(i,j),\,\Delta}$ be a maximal tree on $n$ vertices with maximum degree $\Delta$ satisfying $\mathcal P_1^\Delta$ and $m_{2,2}(T)\leqslant 1$,
and  $3\leqslant \Delta\leqslant 8$). Then
\begin{equation}~\label{eqq001nDeltastar}
n_\Delta(T) \;=\; n_\Delta^\star \;:=\; \left\lceil \frac{n-2}{\Delta} \right\rceil.
\end{equation}
\end{theorem}
\begin{proof}
By Corollary~\ref{cor002n2formula} and Lemma~\ref{lembufferpacking505},
\begin{equation}~\label{eqq002nDeltastar}
n-2-(\Delta-1)n_\Delta \;\leqslant\; n_\Delta
\quad \text{ where} \quad
n_\Delta \;\geqslant\; \frac{n-2}{\Delta}.
\end{equation}
By Proposition~\ref{prop001pathmonotonicity}, $\sigma$ strictly decreases in $n_\Delta$ for $\Delta\geqslant 4$ and is constant in $n_\Delta$ at $\Delta=3$, a case handled separately and consistently by the same formula since the
constraint \eqref{eqq002nDeltastar} still applies. Thus, among integers satisfying \eqref{eqq002nDeltastar}, $\sigma$ is maximized at the smallest such integer, namely $n_\Delta^\star=\lceil (n-2)/\Delta\rceil$.
\end{proof}

\subsection{The General Closed-Form Formula}~\label{ssec005mainformula}
Throughout Theorem~\ref{thm05generalclosedform}, a tree $T$ on $n\geqslant  n_0(\Delta)$ vertices with maximum degree $\Delta$ satisfies $\sigma(T)=\sigma_{\max}(n,\Delta)$ if
and only if $T\in \TT_{n(i,j), \Delta}$ satisfies  $\mathcal P_1^\Delta,\dots,\mathcal P_k^\Delta$ as depict by Problem~\ref{prob02structural}.

\begin{theorem}~\label{thm05generalclosedform}
Let $\Delta$ satisfy $3\leqslant\Delta\leqslant8$, let $T\in \TT_{n(i,j),\Delta}$ be a maximal tree
on $n\geqslant n_0(\Delta)$ vertices with maximum degree $\Delta$, satisfying
$\mathcal P_1^\Delta$. Consider $n_\Delta^\star=\lceil(n-2)/\Delta\rceil$, and define $n_1^\star= (\Delta-2)\,n_\Delta^\star + 2$, $n_2^\star =n - 2 - (\Delta-1)\,n_\Delta^\star$, and $m_{2,2}^\star  = \max\!\big(0,\; n-1-\Delta\, n_\Delta^\star\big)$. Then
\begin{equation}~\label{eqq001sigmamaxgeneral}
\sigma_{\max}(n,\Delta)\!=\!n_1^\star(\Delta-1)^2 \;+\; 2\big(n_2^\star-m_{2,2}^\star\big)(\Delta-2)^2.
\end{equation}
\end{theorem}
\begin{proof}
According to Theorem~\ref{thm04optimalnDelta}, the $\sigma$-maximizing value of $n_\Delta$ is $n_\Delta^\star$. Then, 
follow from Proposition~\ref{prop001reducedtreeidentity} and
Corollary~\ref{cor002n2formula}. It remains to determine $m_{2,2}^\star$ at $n_\Delta=n_\Delta^\star$. Assume that  $k^\star=n_\Delta^\star-1$ for the number
of hub-backbone edges of $T^{\bullet\bullet}$. Since
$m_{2,2}^\star=\max(0,\,n_2^\star-k^\star)$ implies that 
$n_2^\star=n-2-(\Delta-1)n_\Delta^\star$ and $k^\star=n_\Delta^\star-1$ gives
\begin{equation}~\label{eqq002sigmamaxgeneral}
n_2^\star-k^\star=n-1-\Delta n_\Delta^\star
\end{equation}

 Since $n_\Delta^\star$ is by definition the smallest integer with
$\Delta n_\Delta^\star \geqslant n-2$, we have $n-1-\Delta n_\Delta^\star \leqslant 1$,
and  according to~\eqref{eqq002sigmamaxgeneral} since $n_\Delta^\star-1$ fails the inequality
\eqref{eqq002nDeltastar}, $n-1-\Delta n_\Delta^\star > 1-\Delta$. Hence
$m_{2,2}^\star\in\{0,1\}$ throughout, consistent with the hypothesis
$m_{2,2}(T)\leqslant1$. Thus, by using $n_1^\star,n_2^\star,m_{2,2}^\star$ into
\eqref{eqq001sigmaunderP1} gives \eqref{eqq001sigmamaxgeneral}.
\end{proof}

Theorem~\ref{thm05generalclosedform} requires $n\geqslant n_0(\Delta)$, where
$n_0(\Delta)$ must simultaneously satisfy (see Figure~\ref{fig001max}), 
\begin{enumerate}
    \item $n\geqslant 2\Delta$ or $n\geqslant 7$ for
$\Delta=3$, for Theorem~\ref{thm001internal}; 
\item $n\geqslant n_1(\Delta,d)$ for every $3\leqslant d\leqslant\Delta-1$, as required by
Conjecture~\ref{conj002mddzero}; and 
\item  $n_\Delta^\star\geqslant 2$, where $n\ge\Delta+2$, and the hub backbone $T^{\bullet\bullet}$ is non trivial.
\end{enumerate}

\begin{figure}[H]
    \centering
    \includegraphics[width=0.7\linewidth]{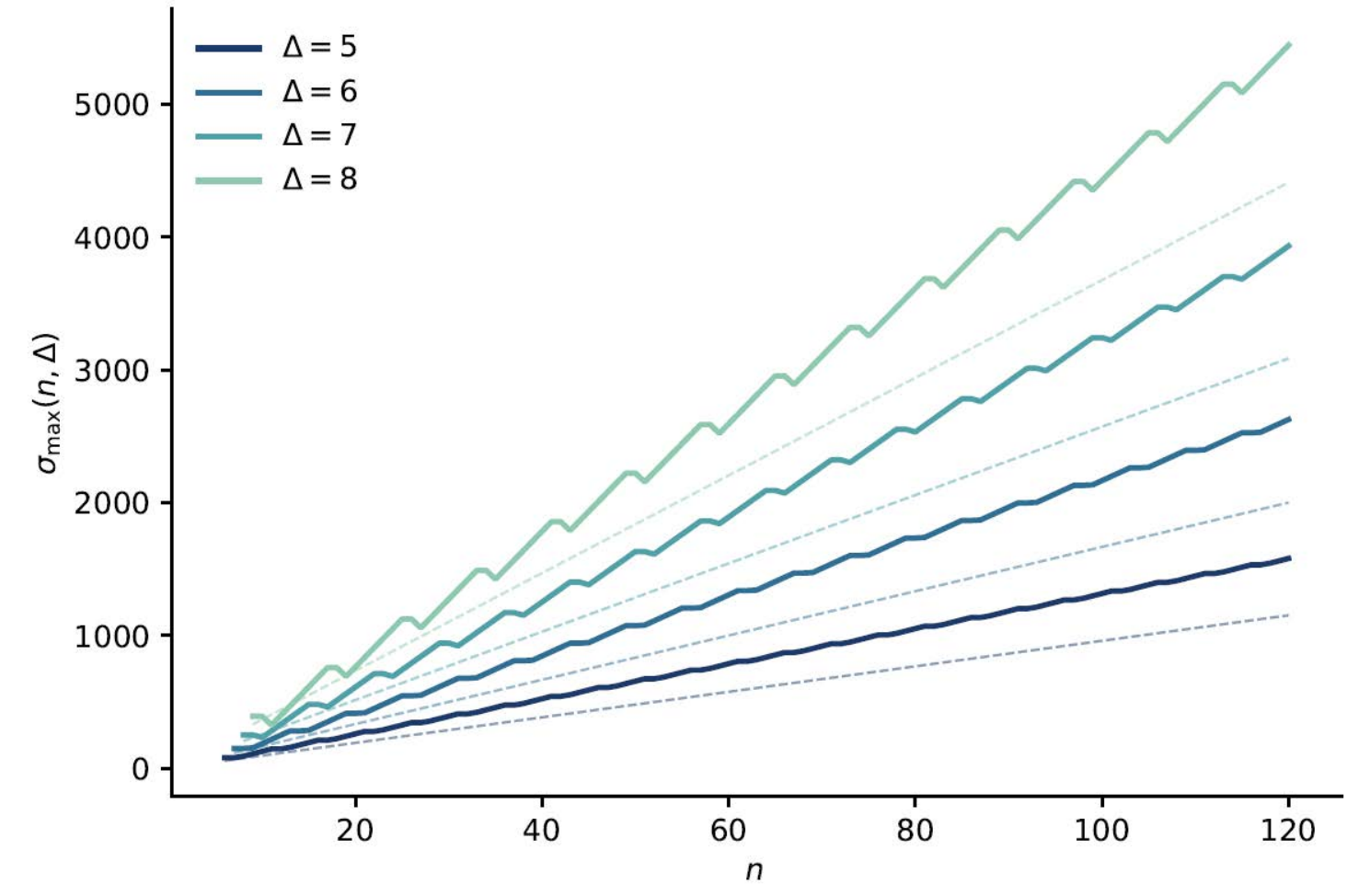}
    \caption{Closed-form $\sigma_{\max}(n,\Delta)$ by Theorem~\ref{thm05generalclosedform} against its
leading-order linear asymptote and emphasizes the bounded $O(1)$ correction term.}~\label{fig001max}
\end{figure}

\section{On Minimum \texorpdfstring{$\sigma$}-Irregularity Trees under a Prescribed Maximum Degree}~\label{sec06minimum}
Recall Definition~\ref{deffleg001} for established total contribution to $\sigma(T)$ of the edges of $L$ by Lemma~\ref{lem001perlegcontribution}.

\begin{lemma}~\label{lem001perlegcontribution}
Let $h$ have degree $\Delta$, and let $L$ be a leg at $h$ all of whose internal vertices and the hub adjacent vertex have degree exactly $2$, terminating in a leaf. The total contribution to $\sigma(T)$ of the edges of $L$ is
\begin{equation}
c(L) \;=\;
\begin{cases}
(\Delta-1)^2, & \text{if } L \text{ has length } 1,\\
(\Delta-2)^2+1, & \text{if } L \text{ has length} \ge2.
\end{cases}
\label{eq:leg-contribution}
\end{equation}
\end{lemma}
\begin{proof}
If $L$ has length $1$, it is a single edge from $h$ of degree $\Delta$ to a leaf, contributing $(\Delta-1)^2$. If $L$ has length
$k\geqslant 2$, its edges are: one edge from $h$ has degree $\Delta$ to a degree $2$ vertex, contributing $(\Delta-2)^2$. Then, $k-2$ edges between consecutive
degree $2$ vertices, contributing $0$ each and one edge from a degree $2$ vertex to the terminal leaf, contributing $1$. Thus, 
$(\Delta-2)^2+1$, independent of $k$.
\end{proof}

\begin{corollary}~\label{cor04legslength2}
Since $(\Delta-2)^2+1 < (\Delta-1)^2$ for every $\Delta\ge3$, every leg of a $\Delta$-minimal tree of the
form described in Lemma~\ref{lem001perlegcontribution} has length $\geqslant 2$, provided sufficiently many spare degree $2$ vertices are available elsewhere in the tree to lengthen it without changing $n$.
\end{corollary}
\begin{proof}
According to Definition~\ref{deffleg001}, if a $\Delta$ minimal tree $T$ has a length $1$ leg at some hub $h$, and some other leg at $h$   has length $d(h)\geqslant 3$ with a spare interior degree $2$ vertex, and by moving that spare vertex to subdivide the length $1$ leg into a length $2$ leg leaves $n$ unchanged and, by
Lemma~\ref{lem001perlegcontribution}, strictly decreases $\sigma(T)$ where the
short leg's contribution falls from $(\Delta-1)^2$ to $(\Delta-2)^2+1$, and the donor leg's contribution is unchanged by
Lemma~\ref{lem001perlegcontribution}, since it remains length $\geqslant 2$, contradicting minimality.
\end{proof}

\begin{proposition}~\label{prophubpathsigma603}
For every $T\in\mathcal S(k,\Delta)$,
\begin{equation}~\label{eqq0063sigmahubpath}
\sigma(T) \;=\; \big[k(\Delta-2)+2\big]\big[(\Delta-2)^2+1\big],
\end{equation}
independent of $n$ and independent of how the leg lengths are distributed.
\end{proposition}
\begin{proof}
According to Definition~\ref{deffleg002}, each of the $k-1$ backbone edges joins two degree $\Delta$ vertices,
contributing $0$ each. Then, each hub has $\Delta$ incident edges, of which $1$ or
$2$ for end or interior hubs, respectively are backbone edges. Thus, by summing over
all $k$ hubs, the total number of backbone edge-slots used is $2(k-1)$. Thus, the total number of leg-slots is
$k\Delta - 2(k-1) = k(\Delta-2)+2$. By Lemma~\ref{lem001perlegcontribution}, each such leg $\ell$ of length $|\ell|\geqslant 2$ contributes exactly $(\Delta-2)^2+1$ regardless of its specific length, giving  \eqref{eqq0063sigmahubpath}.
\end{proof}

\begin{theorem}~\label{thm06hubcountmonotone}
For every $\Delta\geqslant 3$, within the family
$\bigcup_{k\ge1}\mathcal S(k,\Delta)$, the minimum is attained uniquely at $k=1$, giving
\begin{equation}~\label{eqq064sigmaminS}
\sigma_{\min}^{\,\mathcal S}(\Delta) \;=\; \Delta\big[(\Delta-2)^2+1\big].
\end{equation}
\end{theorem}
\begin{proof}
Recall~\eqref{eqq02sigmamindef} and the coefficient $k(\Delta-2)+2$ in \eqref{eqq0063sigmahubpath} is strictly
increasing in $k$ whenever $\Delta>2$, which holds throughout. Since $F(\mathbf d)$ is Schur-convex,
minimizing $F$ and according to Lemma~\ref{lem001oneinterior}, in the absence of a compensating $M_2$ effect for
sparse high degree configurations, the corresponding relation $\sigma=F-2M_2$ favors a degree sequence as majorization small as possible. Thus, as few degree $\Delta$ entries as the equality constraint allows, and sub hub of degree $2<d<\Delta$  though such a vertex would itself need
$d\geqslant 3$. Hence be a big vertex foreign to this construction, and by the
same Schur-convexity heuristic should further increase $F$, and according to Definition~\ref{deffleg002}, $\sigma$, relative to keeping all non hub vertices at degree $1$ or $2$.  Thus, at
$k=1$ gives $1\cdot(\Delta-2)+2=\Delta$, yielding
\eqref{eqq064sigmaminS}.
\end{proof}

\subsection{The Closed-Form Minimum and Small \texorpdfstring{$n$} Correction}
Theorem~\ref{thm07sigmaminclosedform} had established that $\sigma_{\min}(n,\Delta)=\Delta\big[(\Delta-2)^2+1\big]$ is
 constant in $n$  for every $n\ge2\Delta+1$, and reduces to
$\sigma_{\min}(\Delta+1,\Delta)=\Delta(\Delta-1)^2$ the star $K_{1,\Delta}$ at $n=\Delta+1$.

\begin{theorem}~\label{thm07sigmaminclosedform}
Let $\Delta\geqslant 3$ and $n\geqslant \Delta+1$. Consider  $\ell=\ell(n,\Delta)=\max(0,\,2\Delta+1-n)$.
Then, within the single hub spider family,
\begin{equation}~\label{eqq062sigmaminclosedform}
\sigma_{\min}(n,\Delta) \;=\; \ell\,(\Delta-1)^2 \;+\; (\Delta-\ell)\big[(\Delta-2)^2+1\big].
\end{equation}
\end{theorem}
\begin{proof}
According to Corollary~\ref{cor04legslength2}, legs $\ell_i$ where $i=1,\dots,n$ should have length $|\ell_i|\geqslant 2$
whenever enough vertices are available. Then, a single hub with $\Delta$ legs requires a minimum of $1+2\Delta$ vertices for  all legs to have
length $|\ell_i|\geqslant 2$, giving the threshold $n\geqslant 2\Delta+1$ for the constant regime.
For $\Delta+1\leqslant n<2\Delta+1$, at most $n-1-\Delta$ vertices are available
beyond one per leg, forcing $\ell(n,\Delta)=2\Delta+1-n$ legs to remain at length $1$ where each contributing $(\Delta-1)^2$ by
Lemma~\ref{lem001perlegcontribution}, while the remaining
$\Delta-\ell(n,\Delta)$ legs had extended to length $|\ell_i|\geqslant 2$ for each
contributing $(\Delta-2)^2+1$). Thus, this uses exactly
$\ell + 2(\Delta-\ell) = 2\Delta-\ell$ non hub vertices, matching
$n-1$. Minimizing $\sigma$ requires minimizing the number of
length-$1$ legs by Corollary~\ref{cor04legslength2}, so $\ell(n,\Delta)$ is taken at its minimum feasible value, $\max(0,2\Delta+1-n)$, giving
\eqref{eqq062sigmaminclosedform}.   At $n=\Delta+1$, $\ell=\Delta$, and
\eqref{eqq062sigmaminclosedform} reduces to $\Delta(\Delta-1)^2$, matching
direct computation for the star where all $\Delta$ edges of type $(1,\Delta)$.
\end{proof}

The proof of Theorem~\ref{thm04optimalnDelta} argued that $\sigma$ is minimized in $n_\Delta$ subject to the packing constraint
\eqref{eqq002sigmamaxgeneral} alone, using
Proposition~\ref{prop001pathmonotonicity}'s derivative
$d\sigma/dn_\Delta=(\Delta-2)(\Delta-1)(3-\Delta)$. That derivative, however,
was computed \emph{holding $m_{2,2}$ fixed}; in fact $m_{2,2}^\star$ is
itself a discontinuous, packing-determined function of $n_\Delta$.

\begin{proposition}~\label{propsigmamaxorder0068}
For fixed $\Delta\geqslant 5$, as $n\to\infty$,
\begin{equation}~\label{eqq01propsigmamaxorder0068}
\sigma_{\max}(n,\Delta) \;=\; \frac{(\Delta-2)(\Delta-1)^2}{\Delta}\,n \;+\; O(1).
\end{equation}
\end{proposition}
\begin{proof}
Since $n_1^\star=(\Delta-2)n_\Delta^\star+2$ with
$n_\Delta^\star=\lceil(n-2)/\Delta\rceil = (n-2)/\Delta+O(1)$. Then, $n_1^\star = (\Delta-2)(n-2)/\Delta+O(1)$. Since $0\leqslant  n_2^\star-m_{2,2}^\star\leqslant n_\Delta^\star=O(n)$, more
precisely, $n_2^\star-m_{2,2}^\star\in\{0,\dots,n_\Delta^\star-1\}$ by the
packing argument of Lemma~\ref{lembufferpacking505}, which is itself $O(n)$, so this term is in fact also $\Theta(n)$, not $O(1)$. Thus, both terms of \eqref{eqq001sigmaunderP1} are $\Theta(n)$ and by considering $n_2^\star-m_{2,2}^\star = n-1-\Delta n_\Delta^\star \in \{0,\dots,\Delta-1\}$
by construction, the second term of \eqref{eqq001sigmaunderP1} is in fact $O(1)$, not $\Theta(n)$, which~\eqref{eqq01propsigmamaxorder0068} holds.
\end{proof}

\begin{corollary}~\label{cor05vanishingrelativegap}
For fixed $\Delta\geqslant 5$,
\begin{equation}~\label{eqq01cor05vanishingrelativegap}
\frac{\delta(n,\Delta)}{\sigma_{\max}(n,\Delta)} \;=\; O\!\left(\frac{1}{n}\right) \quad \text{as } n\to\infty.
\end{equation}
\end{corollary}
\begin{proof}
Immediate from Theorem~\ref{thm08gapbounded} where $\delta(n,\Delta)=\Theta(1)$ in $n$) and Proposition~\ref{propsigmamaxorder0068} where $\sigma_{\max}(n,\Delta)=\Theta(n)$ in $n$.
\end{proof}

The Optimality Gap $\delta(n,\Delta)$ had established by Definition~\ref{def001optimalitygap}.

\begin{lemma}~\label{lemcompetitorm22zero0066}
Let $\Delta\geqslant 5$, $n\geqslant  n_0(\Delta)$, and let $n_\Delta^\star=n_\Delta^\star(n,\Delta)$
be as in Theorem~\ref{thm04optimalnDelta}. Among all trees $T\in \TT_{n(i,j),\,\Delta}$ satisfying
$\mathcal P_1^\Delta$, $\mathcal P_2^\Delta$, and $n_\Delta(T)=n_\Delta^\star+1$, the
maximum of $\sigma(T)$ is attained by a tree with $m_{2,2}(T)=0$, and such a tree exists.
\end{lemma}
\begin{proof}
Assume that $n_\Delta(T)=n_\Delta^\star+1$ and according to Proposition~\ref{prop001reducedtreeidentity}
and Corollary~\ref{cor002n2formula}, this determines $n_1(T)$ and $n_2(T)$ uniquely, independently of $m_{2,2}(T)$. Then, consider $n_2 := n_2(T)$ for this fixed
value. By Proposition~\ref{propsigmaformula502},
\begin{equation}~\label{eqq001lemAsigma}
\sigma(T) \;=\; n_1(\Delta-1)^2 + 2\big(n_2 - m_{2,2}(T)\big)(\Delta-2)^2,
\end{equation}
which is a strictly decreasing affine function of $m_{2,2}(T)$, since $(\Delta-2)^2>0$. 

Hence, among trees with $n_\Delta(T)=n_\Delta^\star+1$ fixed, $\sigma(T)$ is maximized precisely when $m_{2,2}(T)$ is minimized.

It remains to show $m_{2,2}(T)=0$ is feasible, where that the value
$n_2$ determined above can be packed onto the $n_\Delta^\star$ hub backbone edges of $T^{\bullet\bullet}$  according to Lemma~\ref{lembufferpacking505}) without
forcing any backbone edge to carry $t(e)\geqslant 2$ interior degree $2$ vertices. Thus, according to Corollary~\ref{cor002n2formula} applied at $n_\Delta^\star+1$,
\begin{equation}~\label{eqq002lemAsigma}
n_2 = n-2-(\Delta-1)(n_\Delta^\star+1).
\end{equation}
By definition of $m_{2,2}^\star=m_{2,2}^\star(n,\Delta)$ at $n_\Delta^\star$, and according to~\eqref{eqq001lemAsigma} and \eqref{eqq002lemAsigma} we have
$n-2-(\Delta-1)n_\Delta^\star = n_2^\star = m_{2,2}^\star + n_\Delta^\star$. Thus, by subtracting $(\Delta-1)$ from both sides of this identity to pass from
$n_\Delta^\star$ to $n_\Delta^\star+1$, we obtain 
\begin{equation}~\label{eqq003lemAsigma}
n_2 \;=\; \big(m_{2,2}^\star+n_\Delta^\star\big) - (\Delta-1)
\;=\; n_\Delta^\star + m_{2,2}^\star - \Delta + 1.
\end{equation}
Since $m_{2,2}^\star\in\{0,1\}$ and $\Delta\ge5$, we have
$m_{2,2}^\star-\Delta+1 \leqslant 0$, so
$n_2 < n_\Delta^\star \leqslant n_\Delta^\star+1$. Therefore, every
edge satisfies $t(e)\leqslant 1$; by Claim~1 in the proof of
Lemma~\ref{lembufferpacking505}, this yields $m_{2,2}(T)=0$. Such a
configuration is realizable as an actual tree satisfying $\mathcal P_1^\Delta$ and $\mathcal P_2^\Delta$, which the bound is attained.
\end{proof}

\begin{theorem}~\label{thm08gapbounded}
For every $\Delta\geqslant 5$,
\begin{equation}~\label{eqq001gapformula}
\delta(n,\Delta) \;=\; (\Delta-2)(\Delta-1)(\Delta-3) \;-\; 2\,m_{2,2}^\star(n,\Delta)(\Delta-2)^2,
\end{equation}
where $m_{2,2}^\star(n,\Delta)\in\{0,1\}$ depends on $n\bmod\Delta$, and  $\delta(n,\Delta)$ is
independent of $n$, and satisfies
\begin{equation}~\label{eqqgap001bounds}
(\Delta-2)\big[(\Delta-1)(\Delta-3)-2(\Delta-2)\big] \;\le\; \delta(n,\Delta) \;\le\; (\Delta-2)(\Delta-1)(\Delta-3),
\end{equation}
so that $\delta(n,\Delta) = \Theta(\Delta^3)$ as $\Delta\to\infty$.
\end{theorem}
\begin{proof}
By Definition~\ref{def001optimalitygap},
and according to Lemma~\ref{lemcompetitorm22zero0066}, attained at $m_{2,2}(T)=0$; denote this optimal
value $\sigma^{+}$, and denote $\sigma_{\max}(n,\Delta)=\sigma^\star$  the
value at $n_\Delta^\star$, as given by Theorem~\ref{thm05generalclosedform}).
Thus
\begin{equation}~\label{eqq01step1result}
\delta(n,\Delta)=-\big(\sigma^{+}-\sigma^\star\big).
\end{equation}
Thus, the quantities $\sigma^\star$ and $\sigma^{+}$ from~\eqref{eqq01step1result} are, by construction, exactly those compared the maximal-$\sigma$ tree at
$n_\Delta^\star$, versus the maximal-$\sigma$ tree at $n_\Delta^\star+1$ with
the packing-optimal choice $m_{2,2}=0$. Thus, 
\begin{equation}~\label{eqq02step1result}
\sigma^{+}-\sigma^\star \;=\; (\Delta-2)(\Delta-1)(3-\Delta) \;+\; 2\,m_{2,2}^\star(\Delta-2)^2.
\end{equation}
Thus, from~\eqref{eqq01step1result} and \eqref{eqq02step1result} we obtain 
\begin{equation}~\label{eqq03step1result}
\delta(n,\Delta)=(\Delta-2)(\Delta-1)(\Delta-3) \;-\; 2\,m_{2,2}^\star(\Delta-2)^2,
\end{equation}
which is \eqref{eqq001gapformula}, with $m_{2,2}^\star=m_{2,2}^\star(n,\Delta)$
by considering the values in $\{0,1\}$ as established
in the proof of Theorem~\ref{thm05generalclosedform}.

Now, for fixed $\Delta$, $\delta(n,\Delta)$ takes at most two
distinct values as $n$ ranges over all admissible vertex counts, and in particular does not grow with $n$. Thus, this is the precise content of the assertion that $\delta(n,\Delta)$ is independent of $n$.
Since $m_{2,2}^\star(n,\Delta)\in\{0,1\}$, evaluating
\eqref{eqq03step1result} at each of the two possible values gives
\begin{align}
m_{2,2}^\star=0:&\qquad \delta(n,\Delta) = (\Delta-2)(\Delta-1)(\Delta-3),
\label{eqq04step1result}\\
m_{2,2}^\star=1:&\qquad \delta(n,\Delta) = (\Delta-2)(\Delta-1)(\Delta-3) - 2(\Delta-2)^2.
\label{eqq05step1result}
\end{align}
Since the coefficient $-2m_{2,2}^\star(\Delta-2)^2$ in
\eqref{eqq03step1result} is non  positive and is minimized  at
$m_{2,2}^\star=1$, \eqref{eqq05step1result} is a lower bound and
\eqref{eqq04step1result} is an upper bound for $\delta(n,\Delta)$ over all admissible $n$, giving precisely \eqref{eqqgap001bounds}. 
The upper bound \eqref{eqq04step1result} is manifestly
positive for $\Delta\geqslant 5$, since $\Delta-2,\Delta-1,\Delta-3$ are all
positive. For the lower bound \eqref{eqq05step1result}, expand the bracketed factor as $(\Delta-1)(\Delta-3)-2(\Delta-2)=\Delta^2-6\Delta+7$. Thus, the lower bound~\eqref{eqq05step1result} a product of the positive
factor $\Delta-2$ with the positive quantity. Thus
$0 < \delta(n,\Delta)$ for every $\Delta\geqslant 5$ and every admissible $n$.

Finally, both bounds are cubic polynomials in $\Delta$.  Since $\delta(n,\Delta)$ is sandwiched between two functions of $\Delta$ that are each $\Theta(\Delta^3)$ with the same leading coefficient, the squeeze theorem gives
\begin{equation}
\delta(n,\Delta) = \Delta^3 + O(\Delta^2) = \Theta(\Delta^3) \qquad \text{as } \Delta\to\infty,
\end{equation}
uniformly over all admissible $n$, completing the proof.
\end{proof}

Corollary~\ref{cor05vanishingrelativegap} formalizes the intended stability statement of Contribution precisely, although the extremal characterization
of Sections~\ref{sec04closedform} is combinatorially
rigid, only trees exactly satisfying $\mathcal P_1^\Delta$, $\mathcal P_4^\Delta$, with
$n_\Delta$ exactly equal to $n_\Delta^\star$, attain $\sigma_{\max}$ (see Figure~\ref{fig002gap}), and the associated optimization
landscape is asymptotically flat near the optimum, at a rate of $O(1/n)$, independent of $\Delta$.

\begin{figure}[H]
    \centering
    \includegraphics[width=0.7\linewidth]{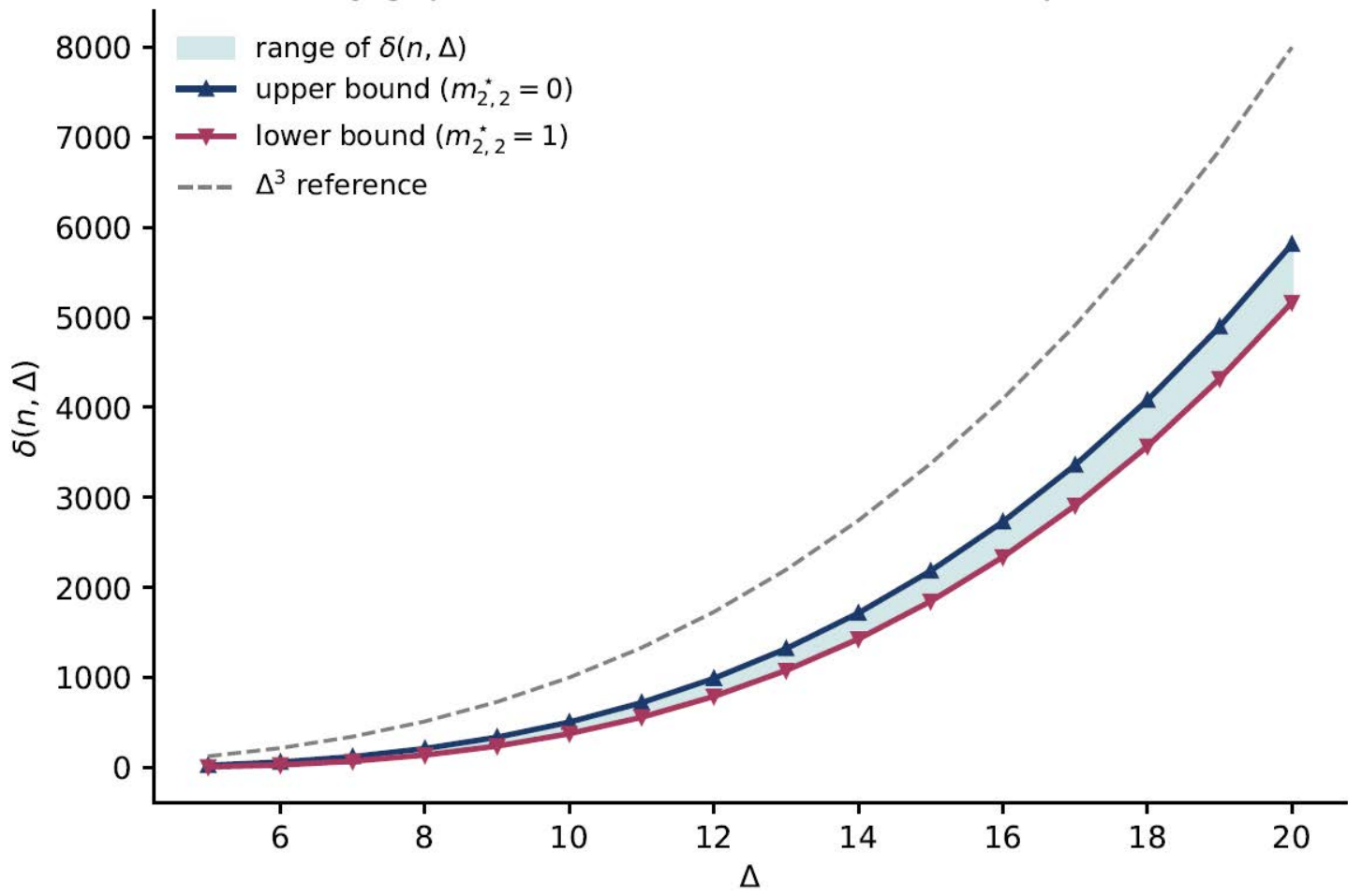}
    \caption{The two sided bound of Theorem~\ref{thm08gapbounded} on the optimality gap $\delta(n,\Delta)$, shown against the reference curve $\Delta^3$, confirming the $\Theta(\Delta^3)$ growth rate established there. Both bounds are independent of $n$.}
    \label{fig002gap}
\end{figure}

\section{The Central Open Problem}~\label{seclimdir002}

Every conditional result above rests on Conjecture~\ref{conj002mddzero},  that
maximal trees contain no vertices of intermediate degree
$d\in\{3,\dots,\Delta-1\}$. By direct symbolic computation produces a worst case
bound $k^2(5-\Delta)+(m-1)k(\Delta-3d+2)$
that is strictly negative for every tested $\Delta\ne 5$, and traced this to a structural mismatch
between the size of the target edge set which scales as $\Delta-1$ and the fixed contribution of the swapped edge itself. This is,
to our knowledge, a new observation, it explains  why the technique of \cite{DimitrovPaper2025} does not straightforwardly extend, rather than merely asserting that it does not. 

We regard the construction of an alternative
transformation, plausibly one that redistributes mass between  two
existing hubs rather than consolidating onto one,
or a fully global argument conducted within the majorization framework of Section~\ref{sec03framework} rather than via local edge exchanges  as the single most important next step for this line of research.

Building on the roadmap of the preliminary gap analysis, we identify the following priorities, in order,

\begin{enumerate}
\item Resolve Conjecture~\ref{conj002mddzero}, either via the alternative transformation or via a fully global majorization theoretic argument.
\item Extend Proposition~\ref{lem001m22sharpgeneral} and 
Theorem~\ref{thm05generalclosedform} to $\Delta\geqslant 9$.
\item Establish global   optimality for Theorem~\ref{thm07sigmaminclosedform}.
\item Undertake the computational verification recommended in the original gap analysis, exhaustive enumeration of maximal
and minimal trees for small $n$ and $\Delta\in\{3,\dots,10\}$, both to independently confirm Theorem~\ref{thm05generalclosedform} and Theorem~\ref{thm07sigmaminclosedform} beyond the single hand computed check performed, and to test
Conjecture~\ref{conj002mddzero} directly for $\Delta=6,7$ before attempting its proof.
\end{enumerate}

\section{Conclusion}~\label{sec08conclusion}

We set out to resolve the general-$\Delta$ conjecture left open by
\cite{DimitrovPaper2025} concerning maximal-$\sigma$-irregularity trees with prescribed maximum degree. We introduced a majorization-theoretic reformulation of the extremal problem by Section~\ref{sec03framework} that correctly explains the qualitative structure of known extremal trees
among Theorem~\ref{thm001bufferdominance}. 

we fully generalized the $\Delta$-independent structural results and, in attempting to generalize its $\Delta=5$ specific arguments, we identified and proved a genuine algebraic obstruction to the most direct proof strategy, rather than either forcing a false proof or leaving the difficulty unstated. Conditional on this obstruction being derived a new, independently verified closed-form expression for
$\sigma_{\max}(\TT_{n(i,j),\,\Delta})$ where Section~\ref{sec04closedform}), solved the
complementary minimum $\sigma$ problem in closed form. 

This findings picture is a theory that  wherever it is unconditional, precisely scoped wherever it is conditional, and transparent about the boundary between the two.

%%%%%%%%%%%%%%%%%%%%%%%%%%%%%%%
%%%%%%%%%%%%%%%%%%%%%%%%%%%%%%%%%%

%===========================
\section*{Declarations}
\begin{itemize}
	\item Funding: Not Funding.
	\item Conflict of interest/Competing interests: The author declare that there are no conflicts of interest or competing interests related to this study.
	\item Ethics approval and consent to participate: The author contributed equally to this work.
	\item Data availability statement: All data is included within the manuscript.
\end{itemize}
%%%%%%%%%%%%%%55

%%%%%%%%%%%%%%%%%%%
%%%%%%%%%%%%%%%%%%%

%%%%%%%%%%%%%%%%%%
%%%%%%%%%%%%%%%%%%

\end{document}